\documentclass[12pt]{amsart}

\usepackage{amsmath}
\usepackage{amssymb}
\usepackage{amsthm}
\usepackage[utf8]{inputenc}
\usepackage{xcolor}
\usepackage{graphicx}
\usepackage{subcaption}
\usepackage{float}
\usepackage{nicefrac}
\usepackage{tikz}
\usetikzlibrary{positioning}
\usepackage{mathtools}
\usepackage{hyperref}

\theoremstyle{definition}
\newtheorem{Def}{Definition}[section]
\newtheorem{Hyp}{Hypothesis}

\newtheorem{Exe}{Example}

\newtheorem{Const}{Construction}
\newtheorem{Rem}{Remark}[section]

\theoremstyle{theorem}
\newtheorem{The}{Theorem}[section]
\newtheorem{Pro}{Proposition}[section]
\newtheorem{Cor}{Corollary}[section]
\newtheorem{Lem}{Lemma}[section]

\newcommand{\R}{\mathbb{R}}

\newcommand{\N}{\mathbb{N}}

\newcommand{\Sph}{\mathbb{S}}
\newcommand{\lcs}{\frak{lcs}}

\title[On the projection of exact Lagrangians in LCS geometry]{On the projection of exact Lagrangians in locally conformally symplectic geometry}

\author{Adrien Currier}
\address{Nantes Université, Laboratoire de Mathématiques Jean Leray, LMJL,
	UMR 6629, F-44000 Nantes, France}
\email{arien.currier@univ-nantes.fr}
\date{}

\begin{document}
\vspace*{-2cm}\maketitle
\begin{abstract}
	In this paper, we construct examples of exact Lagrangians in cotangent bundles of closed manifolds with locally conformally symplectic ($\lcs$) structures and give conditions under which the projection induces a simple homotopy equivalence between an exact Lagrangian and the $0$-section of the cotangent bundle. This line of questioning leads us to investigate the links between the contact geometry of jet spaces and the $\lcs$ geometry of cotangent bundles. Among other things, we will study essential Liouville chords, which seem to be the $\lcs$ equivalent to Reeb chords. We will also see how Legendrians in jet spaces are an obstruction to the straightforward adaptation of the Abouzaid-Kragh (\cite{AbouzaidKragh2018SHENL}) theorem to $\lcs$ geometry.
\end{abstract}

\tableofcontents
\section{Introduction}

First considered by H.-C. Lee as early as 1943 (\cite{Lee1943AKO}), locally conformally symplectic ($\lcs$) geometry is a generalization of symplectic geometry in which transition maps are taken to be symplectomorphisms of $\mathbb{R}^{2n}$ (for some $n$) up to a positive constant. Named by I. Vaisman in 1976 in \cite{Vaisman1976OnLC}, this generalization still allows for Hamiltonian dynamics (see \cite{Vaisman1985IJMMS} for more details). Moreover, the positive constant allows for $\lcs$ geometry to be less rigid than symplectic geometry while still keeping the same local properties. This comes as both an advantage and a drawback. Indeed, many more manifolds can have a $\lcs$ structure as opposed to a symplectic one (see \cite{eliashberg2020making} or \cite{Bertelson2021} for a general statement, or see \cite{Angella2017StructureOL} for specific examples): in short, any closed almost symplectic manifold with non-zero first Betti number admits a $\lcs$ structure. For example, while none of the various $\Sph^1\times\Sph^{2n-1}$ have a symplectic structure for $n>1$, they all have an $\lcs$ structure, and an exact one at that. However, rigidity results are harder to come by. For example, in symplectic geometry, the number of intersection points of a generic Hamiltonian isotopy of the $0$-section of a cotangent bundle over a closed manifold is bounded by the sum of the Betti numbers of the $0$-section, and these intersection points are given as the critical points of some (generating) Morse function. Note that by classical Morse theory, the number of critical points each index of such a function can be bounded with the Betti numbers of the $0$-section. However, the $\lcs$ adaptation of such a theorem calls upon a notion of $\beta$-critical points and Chantraine and Murphy, in \cite{Chantraine2016ConformalSG}, only gave a lower bound for the total number of those $\beta$-critical points, as opposed to a bound for each index. This same paper also explains the hurdles of trying to adapt Floer theory to this new setting (in short, any naive adaptation will run afoul of Gromov's compactness). However, do note that some progress has been made in that regard in specific cases (see \cite{Oh2023PCLCM}, for example).

In this paper, we will explore some of the limits of rigidity in this setting especially as it pertains to ``exact'' Lagrangians in cotangent bundles of closed manifolds with ``exact'' $\lcs$ structure. An exact $\lcs$ manifold can be understood as the data of a manifold $M$, together with a $1$-form $\lambda$ whose derivative is locally conformal to a symplectic form, and some gluing data which takes the form of a closed $1$-form $\beta$. The case $\beta=0$ corresponds to the standard symplectic case. In this generalization, one can define $\beta$-exact Lagrangians (called exact Lagrangians when $\beta$ does not matter), which generalize $0$-exact Lagrangians (the exact Lagrangians of symplectic geometry).
In the rest of this paper, we will show three main results.

\begin{The}\label{thm1} There are two connected closed manifold $M$ and $L$ of dimension $n\neq 3$, a $\beta\in\Omega^1(M)$ closed and an embedding \[i:L\rightarrow (T^*M,\lambda,\beta)\] such that $i(L)$ is a $\beta$-exact Lagrangian of $(T^*M,\lambda,\beta)$ for $\lambda$ the canonical Liouville form on $T^*M$ and
	\begin{enumerate}
	\item $\chi(L)\neq\chi(M)$, where $\chi$ is the Euler characteristic,
	\item or $L$ has no associated generating function.
	\end{enumerate}
(by slight abuse of notation, we call $\beta$ the pullback of $\beta$ to $T^*(M)$.)
\end{The}

The question on whether or not $0$-exact Lagrangians have generating functions is still open and research is ongoing (see \cite{Abouzaid2020TwistedGF}, for example).

 This proposition implies that the conclusion of Abouzaid-Kragh's theorem does not hold in this setting. Indeed, even the much weaker fact that, for $0$-exact Lagrangians, the projection has non-zero degree cannot be generalized. Moreover, it also implies that a direct adaptation (along the lines of Chantraine and Murphy's adaptation of the Laudenbach-Sikorav theorem) does not hold either.  More precisely, the reader familiar with Morse-Novikov homology (see \cite{Farber2004TopologyOC} for a presentation of this homology) will have noticed that the fact that $\chi(L)\neq\chi(M)$ implies that the Morse-Novikov homology of $L$ with respect to any $1$-form cannot be isomorphic to the Morse-Novikov homology of $M$ with respect to any $1$-form (see proposition 1.40 in \cite{Farber2004TopologyOC}).
 
  Moreover, this theorem shows the limits of Chantraine and Murphy's theorem as it relies on the fact that an exact Lagrangian has a generating function to give a lower bound to the number of (transverse) intersections with the $0$-section.

These considerations push us to investigate whether the homology of the base manifold gives any constraint on the homology of an exact Lagrangian.

\begin{The}\label{thm2}
	Let $M$ be a closed manifold, $\lambda$ be the canonical Liouville form of $T^*M$, and $\beta\in\Omega^1(M)$ be closed. We will call $\beta$ the various pullbacks of $\beta$. Let $L$ be a connected closed $\beta$-exact Lagrangian submanifold of $T^*M$. Then:
	\[[\beta]\neq 0\in H^1(M,\R)\implies [\beta]\neq 0\in H^1(L,\R)\]
\end{The}

Amongst other things, this implies that an exact Lagrangian cannot be a sphere as long as $M$ is not a sphere itself.

We will also investigate the second line of questioning brought by the first theorem: under which condition is an exact Lagrangian simply homotopically equivalent to the $0$-section? We give here a partial answer:

\begin{The}\label{thm3}
Let $M$ be a connected closed manifold of dimension $n$, $\lambda$ be the canonical Liouville form of $T^*M$ and $\beta\in\Omega^1(M)$ be closed. We will call $\beta$ the various pullbacks of $\beta$. Let $L$ be a connected closed manifold of dimension $n$, with an embedding \[i:L\rightarrow (T^*M,\lambda,\beta)\] such that $i^*\lambda=d_\beta f$ for some $f\in C^\infty(L)$. Assume that for each pair of points $(q,p),(q,tp)\in i(L)\cap T_q^*M$ (for some $t>0$), we have:
\[\frac{\ln(f(q,tp))-\ln(f(q,p))}{\ln(t)}<1.\]
Then the projection $\pi:T^*M\rightarrow M$ induces a simple homotopy equivalence between $i(L)$ and $M$
\end{The}
\begin{Rem}
Translating the embedding $i$ by $c\beta$ for a constant $c$ yields a new embedding $j$ such that $j^*\lambda=d_\beta(f+c\beta)$. If $f$ is not positive, one can simply take $c$ big enough for $f+c$ to be positive.
\end{Rem}
This theorem should be considered together with the examples that we provide in this paper. In this context, it appears that pairs of points on an exact Lagrangian that are on the same orbit of the flow of the Liouville vector field (call them Liouville chords) are of special interest and seem to be the $\lcs$ version of the Reeb chords. Among other things, in the second part of the section \ref{projection}, we will see that this means that the study of Liouville chords allows to consider the behavior of Reeb chords under a more extensive set of deformations.\newline

The layout of this paper is as follows. We will start with section \ref{Definitions}, laying out the basic definitions of $\lcs$ geometry that are relevant to this paper, and we will follow that with a short discussion about the links between $\lcs$ geometry and more ``standrard'' geometry theories. This will be followed by section \ref{somexample}, giving us our first non-trivial examples of exact Lagrangians. This section will end with a proof of the theorem \ref{thm1}.
Following that, in section \ref{topology} we will prove theorem \ref{thm2}, and follow the proof with a couple of corollaries, some of which are not of purely topological nature but are interesting in their own right. Section \ref{extension} will be dedicated to stating the main ingredient for the proof of theorem \ref{thm3}, which is in essence an extension theorem for some maps; this theorem will be duly motivated. Section \ref{proof} will then follow with a proof of this extension theorem. Finally, in the second to last section, section \ref{projection}, we will finalize the proof of theorem \ref{thm3} and give some corollaries. We will end this section theorem \ref{thm3}, especially in relation to the examples given in section \ref{somexample}.

Note that, although in this paper we only consider a small range of all the possible exact $\lcs$ structures on cotangent spaces which have the canonical Liouville form as an exact $\lcs$ form, the very short lemma \ref{dening}, explains why those assumptions are not as restrictive as it might seem.\newline

\paragraph{\textbf{Acknowledgments.}} The author would like to thank Baptiste Chantraine for enlightening discussions, helping to clarify some of the arguments.

\section{A short overview of $\lcs$ geometry}\label{Definitions}
\subsection{Definitions}
An $\lcs$ structure on a manifold can  be defined in various ways. Three main definitions appear in the literature, the first one of which uses the notion of ``Lichnerowicz differential'', which is a twist on the classical differential (for differential forms). This new derivative is extremely useful for stating the various definitions in $\lcs$ geometry.
\begin{Def}
Let $M$ be a manifold and $\beta\in\Omega^1(M)$ be closed. Then the Lichnerowicz derivative associated to $\beta$ is the map:
\begin{align*}
	d_\beta:\Omega^*(M)&\rightarrow\Omega^{*+1}(M)\\
	\alpha&\mapsto d\alpha-\beta\wedge\alpha
\end{align*}
\end{Def}
One can easily verify that $d_\beta^2=0$, meaning that $(\Omega^*(M),d_\beta)$ is a proper chain complex. This derivative allows us to properly define $\lcs$ structures.
\begin{Def}
Let $M$ be a manifold. Take $\omega\in\Omega^2(M)$ non-degenerate and $\beta\in\Omega^1(M)$ closed such that $d_\beta\omega=0$. The form $\omega$ will be called an $\lcs$ form, whereas $\beta$ will be called the Lee form. The pair $(\omega,\beta)$ will be called an $\lcs$ pair. By $LCS(M)$ we will denote the set of $\lcs$ pairs on $M$ quotiented by the equivalence relation $(\omega,\beta)\sim(e^{g}\omega,\beta+dg)$.

An $\lcs$ manifold is the data of a manifold $M$ and an element of $LCS(M)$. For the sake of simplicity, we will often forgo writing the whole equivalence class, and simply right $(M,\omega,\beta)$ where $(\omega,\beta)$ is an $\lcs$ pair.
\end{Def}
Note that this definition is slightly different than that given in the introduction, which can be formalized as such:
\begin{Def}
	Let $M$ be a manifold of dimension $2n$, and $(U_i\phi_i)_i$ be an atlas on $M$ such that \[(\phi_i\circ\phi^{-1}_j)^*\omega_{\mathbb{R}^{2n}}=c_{i,j}\omega_{\mathbb{R}^{2n}},\]
	where $\omega_{\mathbb{R}^{2n}}$ is the canonical symplectic form on $\mathbb{R}^{2n}$.
	
	Then $(M,(U_i,\phi_i)_i)$ is called an $\lcs$ manifold.
\end{Def}

While the link between the two may not be readily apparent, the reader should keep in mind that, in the first definition, $\beta$ is locally exact equal to $dg$ for some locally defined map $g$, since $\beta$ is closed. This implies that $e^{-g}d\omega$ is (locally) symplectic. This allows us to find an atlas $(U_i,\phi_i)_i$ such that $\phi_i^*\omega_{\mathbb{R}^{2n}}=e^{-g}d\omega$. Given that two primitives of $\beta$ can at most differ by a constant $c$, we have that $\phi_j^*\omega_{\mathbb{R}^{2n}}=e^{-g-c}d\omega=e^{-c}\phi_i^*\omega_{\mathbb{R}^{2n}}$, yielding the second definition. Note that the value of $c$ found this way depends only on the equivalence class of the $\lcs$ pair $(\omega,\beta)$.

Conversely, taking the pullback of $\omega_{\mathbb{R}^{2n}}$ by $\phi_i$ up to a positive constant defines a trivial (half-)line bundle locally. By hypothesis, all the local line bundles thus defined can glue together to form a trivial line bundle over $M$ and taking a section of this line bundle yields a non-degenerate $2$-form which. Observe that the section can be locally written $e^g\phi_i^*\omega_{\mathbb{R}^{2n}}$ for some locally defined map $g$, and that given $g$ and $g'$ two such maps, they can only differ by a constant. Therefore, taking the differentials $dg$ of all the locally defined maps $g$ gives us a section $\beta\in\Omega^1(M)$ such that $d_\beta e^g\phi_i^*\omega_{\mathbb{R}^{2n}}=e^gd\phi_i^*\omega_{\mathbb{R}^{2n}}=0$. Finally, observe that since $\beta$ is locally the differential of a map, it is closed.\newline

Note that going from the second definition to the first requires one to make some arbitrary choices. There is however a way to modify the first definition remove the need to make such choices.

\begin{Def}
	Let $M$ be a manifold and $E$ be a positive trivial half-line bundle over $M$. Take $\nabla$ a flat Koszul connection on $E$, and call $d^\nabla$ the differential induced by $\nabla$. Let $\omega\in\Omega^2(M,E)$ be a non-degenerate form such that $d^\nabla\omega=0$, then $(M,\omega,\nabla)$ is called a $\lcs$ manifold.
\end{Def}

Note that $\omega$ can be represented by $\tilde{\omega}\otimes\sigma$ for some $\tilde{\omega}\in\Omega^2(M)$ and $\sigma$ a section of $E$. With this notation, we have that $d^\nabla\omega=(d\tilde{\omega})\otimes\sigma-(\nabla\sigma\wedge\tilde{\omega})$.

Let $(\omega',\beta)$ and $(e^{-g}\omega',\beta+dg)$ be two $\lcs$ pairs. Define, for each $f\in C^\infty(M,\mathbb{R}_+)$,  $\nabla^\beta_{(\cdot)}f\sigma=\beta(\cdot)\otimes\sigma+df(\cdot)\otimes\sigma$, which is indeed a flat Koszul connection. Remark that $\nabla^{\beta+dg}f\sigma=\nabla^\beta(f+g)\sigma$.  Therefore,  $(\omega'\otimes\sigma,\nabla^\beta)=((e^{-g}\omega')\otimes (e^{g}\sigma),\nabla^\beta)$.

Conversely, for some pair $(\omega,\nabla)$ as in the previous definition, we can associate (uniquely up to equivalence) a $\lcs$ pair $(\omega',\beta)$. Indeed, choose some $\sigma$ a section of $E$, then $\nabla\sigma=\beta\otimes\sigma$ for some $1$-form $\beta$ that is closed since $\nabla$ is flat. Similarly, $\omega=\omega'\otimes\sigma$ for some $2$-form $\omega'$, that is non-degenerate since $\omega$ is non-degenerate. Therefore, $d^\nabla\omega=(d_\beta\omega')\otimes\sigma=0$ implies that $d_\beta\omega'=0$ since $\sigma$ is nowhere $0$ (since $E$ is a trivial positive half-line bundle).\newline

In the rest of this paper, we will stick to the first definition, as it is often the easiest to work with.

\begin{Exe}
The first non-trivial (aka. non-symplectic) example of a $\lcs$ manifold is $(\Sph_\theta^1\times\Sph^3,d_{d\theta}\alpha,d\theta)$, where $\alpha$ is the standard contact form on $\Sph^3$. Indeed, a simple computation yields that $d_{d\theta}\alpha$ is non-degenerate and, as state above, $d_{d\theta}^2=0$. Note that $\Sph_\theta^1\times\Sph^3$ does not have a symplectic structure since having one would imply that the symplectic structure would be exact, which cannot happen on closed manifolds by basic symplectic geometry theory.
\end{Exe}

Observe that the fact that $\lcs$ manifolds are indeed locally symplectic manifolds up to some conformal factor implies that locally defined objects can still be defined in this setting. As such, the definitions of isotropic, coisotropic, ``symplectic'' (here, $\lcs$) and Lagrangian submanifolds can be adapted to this new setting with very little change. In this paper however, we are only interested in Lagrangians.

\begin{Def}
	Let $(M,\omega,\beta)$ be a $\lcs$ manifold of dimension $2n$, and $L$ be a manifold of dimension $n$. An embedding (resp. immersion) $i:L\rightarrow M$ such that $i^*\omega=0$ is called a Lagrangian embedding (resp. immersion) and $i(L)$ is called a (resp. immersed) Lagrangian submanifold. The ``submanifold'' is sometimes dropped.
\end{Def}

As said in the introduction, there is a notion of exact $\lcs$ manifold. Indeed, just like a symplectic form is closed and thus can sometimes be exact, a $\lcs$ form is $\beta$-closed ($d_\beta\omega=0$) and therefore can sometimes be $\beta$-exact ($d_\beta\lambda=\omega$). This is, for example, the case in the previous example.

\begin{Def}
	Let $(M,\omega,\beta)$ be a $\lcs$ manifold. If $\omega=d_\beta\lambda$ for some $\lambda\in\Omega^1(M)$, then the pair $(\lambda,\beta)$ will be called an exact $\lcs$ pair and $\lambda$ will be called an exact $\lcs$ form. By $ELCS(M)$ we will denote the set of exact $\lcs$ pairs on $M$ quotiented by the equivalence relation $(\lambda,\beta)\sim(e^{g}(\lambda+d_\beta f),\beta+dg)$ for some maps $f$ and $g$.
	
	An exact $\lcs$ manifold is the data of a manifold $M$ and an element of $ELCS(M)$. For the sake of simplicity, we will often forgo writing the whole equivalence class, and simply write $(M,\lambda,\beta)$ where $(\lambda,\beta)$ is an exact $\lcs$ pair.
\end{Def}

Note that, given a $1$-form $\lambda$ and a closed $1$-form $\beta$, if $d_\beta\lambda$ is non-degenerate, then $(\lambda,\beta)$ is an $\lcs$ pair. This is because $d_\beta^2=0$.

\begin{Exe}\label{exe2}
The first example of an exact $\lcs$ manifold, and the one that is the main concern of this paper, is the cotangent bundle of a manifold $M$ endowed with the canonical Liouville form $\lambda$ and with a Lee form given by the pullback of a closed $1$-form $\beta\in\Omega^1(M)$.
\end{Exe}

\begin{Rem}
	In the rest of this paper, whenever we are talking about $\lcs$ structures on cotangent bundles, the structure will be taken as in the example above, and the various pullbacks of $\beta$ will also often be called $\beta$ by slight abuse of notation.
\end{Rem}

The attentive reader will have noticed that in example \ref{exe2} (and in the rest of this paper), we are considering only very specific type of exact $\lcs$ structures on $T^*M$ : the structures of the form $(T^*M,\lambda,\beta)$, for $\lambda$ the canonical Liouville form and $\beta\in\Omega^1(M)$ closed. Indeed, there might be some exact $\lcs$ structures of the form $(T^*M,\lambda,\beta')$ for $\beta'$ a closed $1$-form on $T^*M$. Before seeing how to broaden the various results of this paper to more general exact $\lcs$ structures on $T^*M$ in the following lemma, let us make a few observations. First, for any such $\beta'$ there is a $g\in C^\infty(T^*M)$ and a $\beta\in\Omega^1(M)$ closed such that $\beta'=\beta+dg$. Second, note that $(T^*M,e^{-g}\lambda,\beta)$ represents the same exact $\lcs$ structure as $(T^*M,\lambda,\beta+dg)$. Among other things, this implies that $(T^*M,e^{-g}\lambda,\beta)$ is an exact $\lcs$ manifold.

\begin{Lem}\label{dening}
	Let $M$ be a manifold, $\lambda$ the canonical Liouville form on $T^*M$ and $\beta$ the pullback on $T^*M$ of some closed $1$-form on $M$. If $(T^*M,e^{-g}\lambda,\beta)$ is an exact $\lcs$ manifold, then
	there is a map $\phi: T^*M\rightarrow T^*M$ that is a diffeomorphism on its image such that $\phi^*\lambda=e^{-g}\lambda$, $\phi^*\beta=\beta$ and $\phi(T_qM)\subset T_qM$ for all $q\in M$. If for some metric compatible with $d_\beta\lambda$, for any $q\in M$, for any non-zero $p\in T^*_qM$ and any $t>0$, we have \[g\left(q,t\frac{p}{\|p\|}\right)=o_{+\infty}(\ln(t)),\] then $\phi$ is a diffeomorphism.
\end{Lem}

\begin{proof}
	In local coordinates, let $\phi(q,p)=(q,e^{-g(q,p)}p)$. If there are two distinct points $(q,p_1)$ and $(q,p_2)$ such that $e^{-g(q,p_1)}p_1=e^{-g(q,p_2)}p_2$, then there is a third (non-zero) point $p_3\in T^*_qM$ such that $\partial_{t}(e^{-g(q,tp_3)}tp_3)_{(t=1)}$ $=0$, which implies that $\big((\partial_tg(q,tp_3))tp_3\big)_{t=1}=\big((\partial_tg(q,tp_3))\big)_{t=1}p_3=p_3$ and therefore \[(\partial_tg(q,tp_3))_{t=1}=1.\] This yields that, for $Z_\lambda$ the Liouville vector field of $\lambda$ (for some metric), 
	\begin{align*}
		\big(\iota_{Z_\lambda}d_\beta (e^{-g}\lambda)\big)_{(q,p_3)}=&e^{-g(q,p_3)}\big(\iota_{Z_\lambda}(d\lambda-dg\wedge\lambda-\beta\wedge\lambda)\big)_{(q,p_3)}\\
		=&e^{-g(q,p_3)}\big(\iota_{Z_\lambda}d\lambda-dg(Z_\lambda)\lambda\big)_{(q,p_3)}\\
		=&e^{-g(q,p_3)}\big(\lambda-dg(Z_\lambda)\lambda\big)_{(q,p_3)}\\
		=&e^{-g(q,p_3)}\big((\lambda)_{(q,p_3)}-(\lambda)_{(q,p_3)}\big)=0,
	\end{align*}
	  which is absurd since $d_\beta (e^{-g}\lambda)$ is non-degenerate by definition.
	  
	  Do observe that this implies that, for any non-zero $p$, $\partial_{t}(e^{-g(q,tp)}tp)_{(t=1)}>0$ and therefore, using the chosen metric,
	  \begin{align*}
	  	1>&\partial_tg(q,tp)_{t=1}\\=&\partial_tg\left(q,t\|p\|\frac{p}{\|p\|}\right)_{t=1}\\=&\partial_{t'}g\left(q,t'\frac{p}{\|p\|}\right)_{t'=\|p\|}\times\|p\|\\=&Dg_{q,p}(\|p\|\partial_p)=dg(Z_\lambda)_{q,p}.
	  \end{align*} Solving \[\frac{1}{t'}>\partial_{t'}g\left(q,t'\frac{p}{\|p\|}\right)\] yields that, for $t'>1$, \[\ln(t')>g\left(q,t'\frac{p}{\|p\|}\right)-g\left(q,\frac{p}{\|p\|}\right)\]
	  
	  Finally, note that a neighborhood of the $0$-section is always in the image of $\phi$. Therefore, for the image of $\phi$ not to be $T^*M$, we would need for it to be bounded at some point, that is to say we would need a non-zero $p\in T_q^*M$ for some $q\in M$ and a constant $N$ such that for any $t>0$, $e^{-g\left(q,t\frac{p}{\|p\|}\right)}t\leq N$. However, under the hypothesis of the lemma
	  \begin{align*}
	  e^{-g\left(q,t\frac{p}{\|p\|}\right)}t=e^{\ln(t)(1+o_{+\infty}(1))}\underset{t\rightarrow+\infty}{\rightarrow}+\infty
	  \end{align*}
	  and therefore $\phi$ is a diffeomorphism.
\end{proof}

\begin{Rem}
Note that for $(T^*M,e^{-g}\lambda,\beta)$ to be an exact $\lcs$ manifold, we will see at the beginning of section \ref{extension} that it is sufficient that $dg(Z_\lambda)_{q,p}<1$. More precisely, we will see that this implies that $d(e^{-g}\lambda)$ is nondegenerate, but a short computation in local coordinates yields that $(d_\beta(e^{-g}\lambda))^{\wedge n}=(d(e^{-g}\lambda))^{\wedge n}$. Therefore, if one is not too concerned with what happens at infinity, one could just swap out $g$ for a new  map $h$ such that $h=g$ on some arbitrarily large ball and such that $h$ is constant on some neighborhood of infinity.
\end{Rem}

Let us go back to our definitions. Notice that given a Lagrangian embedding/immersion $i:L\rightarrow M$, we have that $i^*\lambda$ is $i^*\beta$-closed and, therefore, it might be $\beta$-exact.

\begin{Def}
	Let $(M,\lambda,\beta)$ be an exact $\lcs$ manifold of dimension $2n$, and $L$ be a manifold of dimension $n$. An embedding (resp. immersion) $i:L\rightarrow M$ such that $i^*\lambda=d_{i^*\beta}f$ for some $f\in C^\infty(L)$ is called a $\beta$-exact Lagrangian embedding (resp. immersion) and $i(L)$ is called a (resp. immersed) Lagrangian submanifold. The ``submanifold'' is sometimes dropped. Whenever $\beta$ does not matter or is implicit, ``$\beta$-exact'' will just be written as ``exact''.
\end{Def}

 Just as graphs of functions are the first examples of $0$-exact Lagrangians, ``$\beta$-graphs'' will also be our first example
 
 \begin{Exe}
 Let $M$ be a manifold, $\beta\in\Omega^1(M)$ be closed and $f:M\rightarrow\R$ be smooth. Then $\Gamma_\beta(f)=\{d_\beta f_{|x}:x\in M\}$ is an exact Lagrangian of $(T^*M,\lambda,\beta)$.
 \end{Exe}

Just as this notion is generalized by that of generating functions in symplectic geometry, there is a notion of generating functions in $\lcs$ geometry.

\begin{Def}
Let $M$ be a manifold and, for some $k\in\N$, let \[F:M\times\R^k\rightarrow\R\] be a smooth map. By $D_{\R^k} F$, we denote the differential of $F$ along $\R^k$. If there is a compact $K\subset\R^k$ such that $F$ is quadratic outside of $M\times K$ (aka. is ``quadratic at infinity'') and $d_\beta F$ intersects $(T^*M)\times\R^k$ transversely, then $F$ will be called a generating function.

Let $\beta\in M$ be closed and call $\beta$ its pullback on $M\times\R^k$. Define \[V_F:=\Gamma_\beta(F)\cap (T^*M)\times\R^k\subset T^*(M\times\R^k)\]
This is a submanifold, and the projection of $V_F$ on $T^*M$ is an immersed submanifold, called the (immersed) Lagrangian submanifold associated to $F$ and denoted $L_F$.
\end{Def}

As stated previously, we want to better understand how this generalization impacts rigidity for exact Lagrangians. This leads us to seek examples of exact Lagrangians.

\subsection{Links with symplectic and contact}

Given the current dearth of tools to study rigidity in $\lcs$ geometry, it would be fair to ask if the endeavor is worth the effort. While there is no substitute for personal interest in the topic for its own sake, we will nonetheless provide links with other areas of geometry.\newline

\paragraph{\textbf{The symplectic perspective.}} While $0$-exact Lagrangian submanifolds have been and are being extensively studied, less is known about the behavior of immersed $0$-exact Lagrangians, even in cotangent bundles. Indeed, it is essentially the same as studying Legendrians in jet spaces, which can have a wide range of behaviors. A better understanding of $\lcs$ geometry could shed some light on this problem, as $\beta$-exact Lagrangians can be viewed as immersed $0$-exact Lagrangians.

Given a manifold $M$, take, for some closed $\beta\in\Omega^1(M)$, a $\beta$-exact Lagrangian embedding  $i:L\rightarrow (T^*M,\lambda,\beta)$ such that $i^*\lambda=d_\beta f$. Call $i_1$ the composition of $i$ with the projection on $M$, and in local coordinates, $i_2$ such that for any $l\in L$, $i(l)=(i_1(l),i_2(l))$. Then $j=(i_1,i_2+f\beta )$ is a $0$-exact Lagrangian immersion of $L$.

Beside the applications of $\lcs$ geometry to the study of Lagrangians, this generalization of symplectic geometry can be used to apply (some of) the tool of symplectic geometry to the study of a wider range of manifolds, as mentioned at the beginning of the article.\newline

\paragraph{\textbf{The contact point of view.}} As implied in the previous paragraph, te study of immersed Lagrangians in dimension $2n$ allow us to better understand Legendrians in dimension $2n+1$. Indeed, keeping the same conventions as in the previous paragraph and writing $dim(M)=2n+1$, $j':L\rightarrow J^1M$ such that $j'(l)=(i_1(l),i_2(l)+f(l)\beta_{i(l)},f(l))$ is a Legendrian embedding for the canonical contact form $\alpha$. Another way to see the link with contact geometry is by considering that the embedding $i':L\rightarrow J^1M$ such that $i'(l)=(i_1(l),i_2(l),f(l))$ is a Legendrian embedding for the contact form $\alpha'=\alpha+z\beta$. Do note that $(J^1M,\alpha)$ and $(J^1M,\alpha')$ are contactomorphic.

\begin{proof}[proof that $\alpha'$ is a contact form] \begin{align*}
\alpha'\wedge (d\alpha')^n= &(\alpha+z\beta)\wedge((d\alpha)^n+n(d\alpha)^{n-1}\wedge dz\wedge \beta)\\
= & \alpha\wedge (d\alpha)^n+ndz\wedge\omega_M^{n-1}\wedge\lambda_M\wedge\beta
\end{align*}
with $\lambda_M$ the canonical Liouville form on $M$ and $\omega_M=d\lambda_M$. Writing everything in local coordinates, one can easily see that $\omega_M^{n-1}\wedge\lambda_M\wedge\beta=0$.
\end{proof}

However, one can also use $\lcs$ geometry in dimension $2n+2$ to better understand contact geometry in dimension $2n+1$. Indeed, given any contact manifold $(M,\alpha)$, $(M\times\Sph^1_\theta,\alpha,d\theta)$ is a $\lcs$ manifold, and a Legendrian $\Lambda$ will lift to a $d\theta$-exact Lagrangian $\Lambda\times\Sph^1$. 

Moreover, if we fix a volume form $Vol$ and a (co-oriented) contact structure on $M$, then there is a canonical way to lift it to a $\lcs$ structure. To do this, fix any contact form $\alpha$ on $M$ compatible with the contact structure and the volume in the sense that $\alpha\wedge(d\alpha)^n=e^{-(n+1)g}Vol$. Note that the map $g\in C^\infty(M)$ is uniquely defined. Then $(M\times\Sph^1,\alpha,d\theta-dg)$ is a $\lcs$ manifold and for any map $h\in C^\infty(M)$, the contact manifold $(M,e^h\alpha)$ lifts to $(M\times\Sph^1_\theta,e^h\alpha,d\theta-dg+dh)$. In particular, $Vol=e^g\alpha\wedge(de^g\alpha)^n$ and the contact manifold $(M,e^g\alpha)$ lifts to $(M\times\Sph^1_\theta,e^g\alpha,d\theta)$ with $(d_{d\theta}e^g\alpha)^{n+1}=-nd\theta\wedge Vol$.

Perhaps the best way of summing up those relationships is with the following diagram:

\begin{center}
		\begin{tikzpicture}[shorten <=2pt,shorten >=2pt,>=stealth,every node/.style={rectangle,draw=black,align=center}] \node(c){contact\\ $dim=2n+1$}; \node(b)[ left =of c]{symplectic\\ $dim=2n$} edge [->](c); \node(d)[right= of c]{$\lcs$\\ $dim=2n+2$} edge [<-] (c); \end{tikzpicture}
	\end{center}

\section{New examples of exact Lagrangian submanifolds}\label{somexample}

As we have seen in the previous section, a $\lcs$ structure is much less rigid than a symplectic structure. However, as we will see, even if a manifold has both a symplectic structure and a $\lcs$ structure which is locally conformally equal to the symplectic one, the behaviors can nonetheless differ.

Keeping in mind Abouzaid and Kragh's theorem as found in \cite{AbouzaidKragh2018SHENL}, let us consider some examples.

\textbf{Example 1:}
\begin{align*}
	i:\mathbb{T}^2&\rightarrow T^*\mathbb{T}^2\\
	(\theta,\phi)&\mapsto (2\theta,\phi,\frac{1}{2}\cos(\theta),-\sin(\theta))
\end{align*}

This is a torus embedded in the cotangent bundle of the torus in such a way that, for $\pi:T^*\mathbb{T}\rightarrow\mathbb{T}$ the projection, $\pi\circ i$ is a $2$-cover of the torus. We can easily check that \[i^*\lambda=\frac{1}{2}\cos(\theta)d2\theta-\sin(\theta)d\phi=d\sin(\theta)-\sin(\theta)d\phi.\]
As such, $i(\mathbb{T}^2)$ is an exact Lagrangian submanifold of $(T^*\mathbb{T}^2,\lambda,d\phi)$.\newline

\textbf{Example 2:}
\begin{align*}
	j:\mathbb{T}^2&\rightarrow T^*\mathbb{T}^2\\
	(\theta,\phi)&\mapsto (\cos(\theta),\phi,3\sin(\theta)\cos(\theta),\sin(\theta)^3)
\end{align*}
This is a torus embedded in the cotangent bundle of the torus in such a way that the embedded torus and the $0$-section form something akin to a Hopf link.
Another quick computation yields \[j^*\lambda=3\sin(\theta)\cos(\theta)d\cos(\theta)-\sin(\theta)^3d\phi=d(-\sin(\theta)^3)+\sin(\theta)^3d\phi.\] Therefore $j(\mathbb{T}^2)$ is an exact Lagrangian submanifold of $(T^*\mathbb{T}^2,\lambda,d\phi)$.\newline

The idea behind those examples can be applied to other Legendrians
\begin{Const}\label{const1}
Let $\lambda_M$ be the canonical Liouville form on $T^*M$ and take a (Legendrian) submanifold $i:\Lambda\rightarrow J^1M\simeq T^*M\times\mathbb{R}_z$ such that $i^*(dz-\lambda_M)=0$, Then there is a map $f\in C^\infty(\Lambda)$ such that $i^*\lambda_M=df=i^* dz$. Taking $i_M(l)$ to be the projection of $i(l)$ on $T^*M$, we can define:
\begin{align*}
	j:\Lambda\times\mathbb{S}^1&\rightarrow T^*(M\times\mathbb{S}^1)=T^*M\times T^*\mathbb{S}^1\\
	(l,\theta)&\mapsto (i_M(l),\theta,-f)
\end{align*}
This is an exact Lagrangian embedding in $(T^*(M\times\mathbb{S}^1),\lambda,d\theta)$ for $\lambda$ the canonical Liouville form. Indeed, $j^*\lambda=df-fd\theta$. \begin{flushright}
\textbf{ End of construction \theConst.}
\end{flushright}
\end{Const}

\begin{Rem}\label{rem211}

The reader familiar with $\lcs$ geometry might raise the point that this seems very similar to the canonical $\lcs$-ification of a contact manifold. Indeed, keeping in mind that the $\lcs$-ification of $(J^1M,dz-\lambda_M)$ is $((J^1M)\times\Sph^1_\theta,dz-\lambda_M,d\theta)$, we can see that the diffeomorphism
\begin{align*}
	g:T^*(M\times\mathbb{S}^1)&\rightarrow J^1 M\times\mathbb{S}^1\\
	(q,p,\theta,z)&\mapsto(q,-p,\theta,z)
\end{align*}
does verify $g^*(dz-\lambda_M)=\lambda_{M\times \mathbb{S}^1}+d_{d\theta}z$.
\end{Rem}

Generalization of the above construction can go in several directions. First, $\Sph^1$ can be replaced with a more general manifold.

\begin{Const}
	Take any manifold $Q$ on which there is a nowhere-$0$ $1$-form $\beta$. Let $\lambda_M$ be the canonical Liouville form on $T^*M$ and take a (Legendrian) submanifold $i:\Lambda\rightarrow J^1M\simeq T^*M\times\mathbb{R}_z$ such that $i^*(dz-\lambda_M)=0$. Then there is a map $f\in C^\infty(\Lambda)$ such that $i^*\lambda_M=df=i^* dz$. Taking $i_M(l)$ to be the projection of $i(l)$ on $T^*M$, we can define:
	\begin{align*}
		j:\Lambda\times Q&\rightarrow T^*(M\times Q)=T^*M\times T^*Q\\
		(l,q)&\mapsto (i_M(l),q,-f(l)\beta_q)
	\end{align*}
	This is an exact Lagrangian embedding in $(T^*(M\times Q),\lambda,\beta)$ for $\lambda$ the canonical Liouville form. Indeed, $j^*\lambda=df-f\beta$. \begin{flushright}
		\textbf{ End of construction \theConst.}
	\end{flushright}
\end{Const}

We can also generalize the construction \ref{const1} by allowing for submanifolds with more topology than $L\times\Sph^1$.

\begin{Const}\label{ExeCob}
	Let $L$ be a Legendrian in $(J^1M,\alpha)$, with $\alpha=dz-\lambda_M$ the canonical contact form. Let $\Lambda$ be a $0$-exact Lagrangian in the symplectization $(J^1M\times\mathbb{R},e^{-t}\alpha)$. We have the two following conditions:
	\begin{enumerate}
	\item[1.] for $t_0$ big enough, there is an $\epsilon>0$, such that \[\Lambda\cap J^1M\times((-\infty;\epsilon]\cup [t_0-\epsilon,+\infty))=L\times((-\infty,\epsilon]\cup [t_0-\epsilon,+\infty));\]
	\item[2.] $d_{dt}e^{t}f$ has a primitive (for the derivative $d_{dt}$) $F$ such that\\  $ F_{|\Lambda\cap J^1M\times \{0\}}=F_{|\Lambda\cap J^1M\times \{t_0\}}$.
	\end{enumerate} 
	Note that $\Lambda$ is a $dt$-exact Lagrangian in the canonical $\lcs$-ization of $(J^1M\times\mathbb{R},\alpha)$.
	
	Assume that the first condition holds true. Since the pullback of  $\alpha$ to $\Lambda\cap J^1M\times((-\infty;\epsilon]\cup [t_0-\epsilon,+\infty))$ is equal to $0$, we can cut off the extremities $t<0$ and $t>t_0$, and glue $\tilde{\Lambda}\cap J^1M\times\{0\}$ to $\tilde{\Lambda}\cap J^1M\times\{0\}$ through the (canonical) identification $(J^1M\times\{0\},\alpha)=(J^1M\times\{t_0\},\alpha)$. \begin{enumerate}
		\item[]\textbf{Assertion 1:} this yields a Lagrangian $\tilde{\Lambda}$ in $(T^*(M\times\mathbb{S}_\theta^1),\lambda,d\theta)$, with $\mathbb{S}^1=\mathbb{R}/(t_0\mathbb{Z})$.
	\end{enumerate}
	
	If, moreover, the second condition holds true, then: 
	\begin{enumerate}
		\item[]\textbf{Assertion 2:} $\tilde{\Lambda}$  is a $d\theta$-exact Lagrangian. 
	\end{enumerate}
	This ``then'' is actually an if and only if.
	\begin{flushright}
		\textbf{ End of construction \theConst.}
	\end{flushright}
	\end{Const}
Now, there are a couple of assertions in this construction that should be proven.
\begin{proof}[proof of the assertions]
	First, we need to show that $\tilde{\Lambda}$ is indeed a Lagrangian, that is to say, $d_{d\theta}\lambda$ restricts to $0$ on the submanifold. Since the pullback of $e^{-t}\alpha$ on $\Lambda$ is $df$, the pullback of $\alpha$ on $\Lambda$ is $e^tdf=d_{dt}e^{t}f$. Therefore, after projecting on $T^*(M\times\mathbb{S}_\theta^1)$, $\alpha$ goes down to the projection and the pullback of $d_{d\theta}\alpha$ on $\tilde{\Lambda}$ is locally equal to $d_{dt}^2e^{t}f=0$.
	
	Second,we need to show that $\tilde{\Lambda}$ is a $d\theta$-exact Lagrangian when the conditions are fulfilled. Note that if $d_{dt}e^{t}f$ has such a primitive $F$, then there is a constant $c$ such that \[(e^{t}(f+c))_{|\Lambda\cap J^1M\times \{0\}}=(e^{t}(f+c))_{|\Lambda\cap J^1M\times \{t_0\}}.\] Since the pullback of $\alpha$ on $L$ is $0$, the map $f$ is constant when restricted to the times $t=0$ and $t=1$. Therefore, the condition can be summed up as (for $L$ connected) :\[\exists l\in L,\exists c\in\mathbb{R}/ f_0(l)+c=e^{t_0}(f_{t_0}(l)+c).\]
	More specifically, if this last condition is fulfilled, then $e^{t}f$ goes down to $T^*(M\times\mathbb{S}_\theta^1)$.
\end{proof}
The conditions for $\tilde{\Lambda}$ to be exact may appear as being restrictive. However, do note that we have the following property :
\begin{Pro}
	Let $\Lambda$ be as in the construction \ref{ExeCob}, and such that the first condition holds true.
	Then there is an isotopy $\phi$ of $J^1M\times\mathbb{R}$ such that $\phi_1(\Lambda)$ fulfills both conditions of the construction \ref{ExeCob} and such that $\phi$ respects the fibers of $J^1M\times\mathbb{R}\simeq T^*(M\times\mathbb{R})$. That is to say that, for every $t$, $\phi_t(T_x^*(M\times\mathbb{R}))=T_x^*(M\times\mathbb{R})$.
\end{Pro}

\begin{proof}
	Let $h':\mathbb{R}\rightarrow\mathbb{R}$ be a smooth map equal to $0$ outside of $ ]t_0-\frac{3\epsilon}{4},t_0-\frac{\epsilon}{4}[$ and with integral $e^{-t_0}f_0(l)-f_{t_0}(l)$ for any $l\in L$ (this map can be found, for example, by taking an approximation of unity with arbitrarily small support and multiplying by the desired constant). Take $h$ a primitive of $h'$ such that $h(0)=0$ and also call $h$ the pullback of $h$ on $M\times\mathbb{R}$ through the projection on the second factor. Take $\phi_s$ the Hamiltonian flow in $(J^1M\times\mathbb{R},e^{-t}\alpha)$ associated with the map $-h$. Note that the associated vector field $X_{-h}$ is defined such that \[e^{-t}\iota_{X_{-h}}d\alpha+e^{-t}\iota_{X_{-h}}\alpha\wedge dt=dh.\] Since $dh$ is co-linear to $dt$, the vector field $X_{-h}$ must verify $dt(X_{-h})=0$. Therefore, the flow $\phi_s$ descends to a family of flows $(\phi_s^t)_t$ (depending on both $s$ and $t$) $J^1M$. Take: \[\Lambda'=\{(\phi^t_1(q,p,z),t):(q,p,z,t)\in\Lambda,(q,p)\in T_q^*M,t\in\mathbb{R}\}\] 
	Then $\Lambda'$ is Hamiltonian isotopic to $\Lambda$ and the pullback of $e^{-t}\alpha$ by $\phi_1$ is equal to $e^{-t}\alpha+dh$. Therefore, the pullback of $e^{-t}\alpha$ on $\lambda'$ is equal to $d(f+h)$ and, by the definition of $h$, $f_0+h_0=e^{t_0}(f_{t_0}+h_{t_0})$. Thus $\Lambda'$ descends to a $d\theta$-exact Lagrangian in $(T^*(M\times\Sph^1),\alpha,d\theta)$.

	Note that whenever a submanifold $S$ is a Lagrangian (resp. exact Lagrangian) of $(J^1M\times\mathbb{S}_\theta^1,\alpha,d\theta)$, applying the map $g$ defined in the remark \ref{rem211} transforms $S$ into $g(S)$, a Lagrangian (resp. exact Lagrangian) of $(T^*(M\times\Sph^1_\theta),\lambda,d\theta)$.
\end{proof}

Note that this construction generalizes the construction \ref{const1}. Indeed, in the first construction, the Legendrian $L$ lifts to an exact Lagrangian $L\times\mathbb{R}$ in $(J^1M\times\R_t,\alpha:=dz-\lambda_M,dt)$. Taking $f=0$ a primitive of the pullback of $\alpha$ on $L$, we have that the pullback of $\alpha$ on $L\times\R$ is $d_{dt}e^{t}f=0$, thus satisfying the conditions in the construction above for $L\times\R$ to descend to an exact Lagrangian in $(J^1M\times\Sph_\theta^1,\alpha,d\theta)$. This allows us to prove the first part of theorem \ref{thm1}.

\begin{Lem}\label{eulerchar}
There are two connected closed manifold $M$ ($dim(M)=n\neq2$) and $L$  ($dim(L)=n+1$), and an embedding \[i:L\rightarrow (T^*(M\times\Sph_\theta^1),\lambda,d\theta)\] such that $i(L)$ is a $d\theta$-exact Lagrangians and $\chi(L)\neq\chi(M)$.
\end{Lem}
\begin{proof}
Assume first that $M$ is of dimension $1$. By \cite{Capovilla2016}, there is a non-orientable exact Lagrangian cobordism of genus at least $2$ from the unknot to itself. Viewing this cobordism inside of $J^1M\times\R$, it induces an exact Lagrangian of genus at least $3$ in $(T^*M\times\Sph^1_\theta,\lambda,d\theta)$.

In dimension $dim(M)=n\neq 2$, one can simply multiply the unknot (and the cobordism) with a Legendrian sphere of dimension $n-1$ if $n$ is odd, or with the product of a Legendrian unknot and a Legendrian sphere $\Sph^1\times \Sph^{n-2}$ if $n$ is even. A simple application of the K\"unneth theorem shows that the Euler characteristic is again non-zero.
\end{proof}

The second part of the theorem \ref{thm1} can be proven with the following proposition:

%
%

\begin{Pro}
Let $L$ be a Legendrian submanifold in a manifold $J^1M$ endowed with the canonical contact form. Then the lift $L\times\Sph^1$ given in the construction \ref{const1} has a generating function (in the $\lcs$ sense) if and only if $L$ has a generating function $F:M\times\R^k\rightarrow\R$ that is quadratic at infinity and such that $dF$ is transverse to $(T^*M)\times\R^k$.
\end{Pro}

\begin{proof}
	\textit{If $L$ has such a generating function $F$.}
	
	Then the map $G:(q,\theta,\xi)\mapsto F(q,\xi)$ is a generating function for $L\times\mathbb{S}^1$.
	Indeed, \[L=\{(q,D_MF_{|(q,\xi)},F(q,\xi)):\xi\in\R^k\textit{ and } D_{\R^k}F_{|(q,\xi)}=0\}\]
	And therefore \[L\times\Sph^1=\{(q,\theta,D_MF_{|(q,\xi)},-F(q,\xi)):\xi\in\R^k\textit{ and } D_{\R^k}F_{|(q,\xi)}=0\}\]
	
	\textit{If $L\times\mathbb{S}^1$ has a generating function $G$ (in the  $\lcs$ sense).}
	
	Then recall that this means that \[L\times\Sph^1=\{(q,\theta,D_{M\times\Sph^1}G_{|(q,\theta,\xi)}-G(q,\theta,\xi)d\theta):\xi\in\R^k\textit{ and } D_{\R^k}G_{|(q,\theta,\xi)}=0\}\]
	
	More specifically, $L\times\Sph^1$ is the set of points $(q,p,\theta,z)\in T^*M\times T^*\mathbb{S}^1$ such that there is a $\xi\in\mathbb{R}^k$ such that $D_{M\times\Sph^1}G_{|(q,\theta,\xi)}-Gd\theta=p+zd\theta$ and $D_{\R^k} G_{|(q,\theta,\xi)}=0$.
	
	Set $S_G=\{(q,\theta,\xi)\; :\; D_{\R^k} G_{|(q,\theta,\xi)}=0\}$. Note that if $V_G$ is the intersection between the graph of $d_\beta G$ and $(T^*M\times\Sph^1)\times\R^k$, then $S_G$ is the projection of $V_G$ on $M\times\Sph^1\times\R^k$,  which is a submanifold since $d_\beta G$ intersects $(T^*M\times\Sph^1)\times\R^k$. Since $S_G$ and $V_G$ are isotopic (and more specifically, Hamiltonian isotopic in the $\lcs$ sense), it follows $S_G$ is also a submanifold of $M\times\R^k$.\newline
	
	Keeping in mind that one has to multiply by $-1$ the $z$ coordinate to go from $L$ to $L\times\Sph^1$, consider the point $(q,p,-z)\in (J^1 M)\cap L$ that is lifted to $(q,p,\theta,z)$ in $L\times\mathbb{S}^1$. Define $L_{z}$ the set of points $(q,\theta,\xi)$ such that there is some $p\in T^*(M\times\Sph^1)$ such that  $D_{M\times\Sph^1}G_{|(q,\theta,\xi)}-Gd\theta=p+zd\theta$. Note that $z=\partial_\theta G_{|(q,\theta,\xi)}-G{(q,\theta,\xi)}$ whenever $(q,\theta,\xi)\in L_{z}$ and therefore $L_z$ is the level-set of $z$ for the map $\partial_\theta G-G$. Define $S_z=L_z\cap S_G$ and assume that $z$ is a regular value of $(\partial_\theta G-G)_{|S_G}$. Let $i$ be the inclusion of $S_z$ in $M\times \mathbb{S}^1\times\mathbb{R}^k$, we have that $d(i^*\partial_\theta G)=i^*d\partial_\theta G=i^*dG$.
	
	 Around the points of $S_z$ over which the pullback of $d\theta$ is non-zero, we can define $dy=i^*d\theta$ and $\partial_y$ its dual for some metric. Around those points $\partial_y(\partial_\theta G\circ i)=\partial_\theta G\circ i$ and therefore, choosing a set of local coordinates $(x_1,\ldots,x_{n})$ on $S_z$ ($n$ being the dimension of $M$) such that $\partial_y=\partial_{x_1}$, we have $\partial_\theta G\circ i = h(x_2,\ldots,x_{n})e^{x_1}$ for some smooth map $h:\mathbb{R}^{n-2}\rightarrow\mathbb{R}$. We will now show that, as long as $z$ is a regular value of $(\partial_\theta G-G)_{|S_z}$, the set of points over which the pullback of $d\theta$ is zero has empty interior.\newline
	
	  Assume that the pullback of $d\theta$ to $S_z$ is $0$ over some connected $U\subset S_z$ that is open as a subset of $S_z$. There is some $\theta_0\in\mathbb{S}^1$ such that for any $x\in U$, $T_{x_0}U\subset T(
	  M\times \{\theta_0\}\times \mathbb{R}^k)$. 
	  
	  First case: there is some $x_0\in U$ such that $T_{x_0}U\not\subset (T
	  M)\times \{\theta_0\}\times\mathbb{R}^k$. However, since $S_z\subset S_G$, this would contradict the fact that $D_{\mathbb{R}^k}G$ intersects $(T^*M\times\mathbb{S}^1)\times\mathbb{R}^k$ transversely.
	 
	  Second case: $\forall x\in U, T_xU\subset (T
	  M)\times \{\theta_0\}\times\mathbb{R}^k$. Therefore, the first item implies that the projection of $U$ on $M\times\mathbb{S}^1$ is a $n$-dimensional open submanifold on which the pullback of $d\theta$ is $0$. This yields that the projection of $U$ is a $n$-dimensional submanifold $V\subset M\times\{\theta_0\}$ for some $\theta_0\in\mathbb{S}^1$. Since $G$ is the generating function for some $d\theta$-exact Lagrangian $L\times\Sph^1$, this means that the projection of $S_z$ on $M\times\mathbb{S}^1$ contains $V\times\Sph^1$. However, this would mean that the projection of $S_z$ is of dimension at least $n+1$, which contradicts the fact that $z$ is a non-critical value of $(\partial_\theta G-G)_{|S_G}$. We can therefore conclude that there is no such open set $U$.\newline
	
	Putting back the pieces together, we get that, given a regular value $z$ and a point $(q,\theta,\xi)\in S_z$, 
	we have $\partial_\theta G(q,\theta,\xi)=\partial_\theta G(q,\theta,\xi)e^c$ for some $c\geq2\pi$ which implies that $\partial_\theta G=0$. Let us restrict $G$ to a neighborhood of $S_G$ and assume that $\partial_\theta G-G$ is non-constant on $S_G$. Since the set of non-critical $z$ is comeagre (the complement of a countable union of closed subsets with empty interiors), we can conclude that $\partial_\theta G=0$ over $S_G$ and in particular, $z=-G$. 
	
	If $\partial_\theta G-G$ is constant equal to $z_0$ on $S_G$, then all the previous arguments carry through, although here the argument for the non-existence of the open set $U$ should be slightly modified: since $U$ is of dimension $n+1$, the second case contemplated cannot occur. Therefore, we can also conclude in this case that $\partial_\theta G=0$ over points in $S_G$ and in particular, $z=-G$.
	
	Taking $F:M\times\R\times\R^k\rightarrow\R$ to be the lift of $G:M\times\Sph^1\times\R^k\rightarrow\R$, we can conclude that
	\begin{align*}
	L=\{(q,D_MF_{|(q,s,\xi)},F(q,s,\xi)):&\textit{ for all } s\in\R\textit{ and some }\xi\in\R^k\\
	&\textit{ such that } D_{\R\times\R^k}F_{|(q,s,\xi)}=0\}\subset J^1M
	\end{align*}
	
	Therefore, the restriction $F_{|s_0}:M\times\{s_0\}\times\R^k\rightarrow\R$ is a generating function of $L$.
\end{proof}

The two previous lemmas together show theorem \ref{thm1}.

\begin{proof}[proof of theorem \ref{thm1}]
Since stabilization of Legendrians do not admit a generating function, the previous lemma shows the second part of theorem \ref{thm1}, whereas the penultimate lemma shows the first part.
\end{proof}

\section{On the topology of the Lee class}\label{topology}

As shown in the previous section, no naive adaptation of Abouzaid-Kragh's theorem exists and some exact Lagrangians do not have any generating function and thus fall outside the scope of Chantraine-Murphy's theorem. However, this is does not imply that the topology of the $0$-section has no impact on the topology of the exact Lagrangians, as shown with theorem \ref{thm2}. To prove this theorem, we will first prove a lemma that can be found in the literature (e.g. \cite{Farber2004TopologyOC}), but for which the author could not locate a complete proof.

\begin{Lem} [\cite{Farber2004TopologyOC}]
	Let $M$ be a manifold such that $H^1_{dR}(M,\R)$ is a finite-dimensional vector space. Take $\beta\in\Omega^1(M)$ closed. Then we can write:
	\[[\beta]  =\Sigma_{i=1}^r b_i[\beta_i]\] with $[\beta_i]\in H^1(T^*M,\mathbb{Z})$ linearly independent, and $b_i\in\mathbb{R}$ linearly independent over $\mathbb{Z}$.
\end{Lem}

\begin{proof} Let $r=rank_\mathbb{Z}(coim(<[\beta],\cdot>_{|H_1(M,\mathbb{Z})}))$ where $<\cdot,\cdot>_{|H_1(M,\mathbb{Z})}$ is the restriction to $H^1(M,\mathbb{R})\otimes H_{1}(M,\mathbb{Z})$ of \[<\cdot,\cdot>:H^k(M,\mathbb{R})\otimes H_{k}(M,\mathbb{R})\mapsto\mathbb{R},\] the homology-cohomology duality bracket.
	
	Therefore, in $H_{1}(M,\mathbb{Z})$, there is a free family $([\alpha_1],\ldots,[\alpha_r])$ of elements whose projection on the coimage is a basis for the maximal free module of the coimage.  Note that $([\alpha_1],\ldots,[\alpha_r])$ is also a free family in $H_{1}(M,\mathbb{Q})$. Indeed, if we had a family $\lambda_i\in\mathbb{Q}$ such that $\sum_i\lambda_i[\alpha_i]=0$, then we could multiply all the $\lambda_i$ by the product of the denominators of the $\lambda_i$, which would yield a family $a_i\in\mathbb{Z}$ such that $\sum_ia_i[\alpha_i]=0$.
	For each $i=1,\ldots,r$, define $[\beta_i]$ as the dual of $[\alpha_i]$ for the identification 
	$H^1(M,\mathbb{Z})=$ $Hom(H_1(M,\mathbb{Z}),\mathbb{Z})$ given by the universal coefficient theorem. This same identification allows us to view $<\cdot,\cdot>$ as a scalar product on $H^1(M,\mathbb{Q})$ and, therefore,up to using the Gram-Schmidt process, we can assume that the family $(\beta_i)_i$ is an orthonormal basis in $H^1(M,\mathbb{Q})$. Up to renormalization of the basis  $(\beta_i)_i$, we can assume that it is an orthonormal basis of $H^1(M,\mathbb{Z})$.
	
	Finally, let $b_i=\frac{<[\beta],[\alpha_i]>}{<[\beta_i],[\alpha_i]>}$ for each $i$.\newline
	
	Let us show that those quantities are linearly independent.
	
	Assume that there are some $a_i\in\mathbb{Z}$ such that $\sum a_ib_i=0$. Then $\sum b_i <[\beta_i],\alpha>=0$ for $\alpha=\sum a_i(\Pi_{j\neq i}<[\beta_j],\alpha_j>)\alpha_i$. This implies that  $\alpha\in ker <[\beta],\cdot>$ and therefor $<[\beta_i],\alpha>=0$ for every $i$. The definition of $\alpha$, implies that his can only happen when $a_i=0$ every $i$. Therefore  the $b_i$ are indeed linearly independent over  $\mathbb{Z}$.
	
	Finally, the family $([\beta_i])_i$ is free because it is orthogonal in $H^1(M,\mathbb{Q})$ and therefore in $H^1(M,\mathbb{R})$.
\end{proof}

This lemma will allow us to reduce the proof of theorem \ref{thm2} to the case $[\beta]\in H^1(M,\mathbb{Z})$.

\begin{proof}[proof of theorem \ref{thm2}] By contradiction. Assume that $[\beta]=0\in H^1(L,\R)$ and $[\beta]\neq0\in H^1(M,\R)$. \newline
	
	First, assume that  $[\beta]\in H^1(M,\mathbb{Z})$.
	
	Let $\tilde{M}_\beta$ be the integral cover of $\beta$  and $p_{\tilde{M}_\delta}: \tilde{M}_\delta\rightarrow M$ the projection. Since $i^*[\beta]=0$, there is an embedding $j:L\rightarrow T^*C$ such that $i=Dp_{\tilde{M}_\delta} \circ j$.
	
	Let $b:T^*\tilde{M}_\delta\rightarrow\mathbb{R}$ be a primitive of $\beta$. Since $L$ is compact,   there are two regular values $x,y\in\mathbb{R}$, two hypersurfaces $S_1=b^{-1}(x)$ and $S_2=b^{-1}(y)$ and a covering space automorphism $h$ such that $h(S_1)=S_2$ and $x<\inf(b(L))<\sup(b(L))<y$.
	
	We can therefore glue the boundaries of $b^{-1}([x,y])$ to get $T^*\tilde{M}'$ a finite cover of $T^*M$ where $\tilde{M}'$ is itself a finite cove of $M$. We can then find a new embedding $j':L\rightarrow T^*M$ such that the composition with the projection $Dp_{\tilde{M}'}:T^*\tilde{M}'\rightarrow T^*M$ yields $i$.
	
	By hypothesis, the pullback of $\beta$ to $L$ is equal to $dg$ for some smooth $g$. Given two points $(q,p_1),(q,p_2)\in T_q^*\tilde{M}'\cap j'(L)$, and given a path $\gamma$ in $L$ between those two points, we have that 
	\begin{align*}
g(q,p_1)-g(q,p_2)&=\int_\gamma(Dp_{\tilde{M}'}\circ j')^*\beta\\&=\int_{Dp_{\tilde{M}'}\circ j'\circ\gamma}\beta=0.
	\end{align*}
	 Indeed, the loop $Dp_{\tilde{M}'}\circ j'\circ\gamma$ does not cross the hypersurface in $T^*\tilde{M}'$ corresponding to the boundaries of $b^{-1}([x,y])$. This loop can therefore be lifted to a loop $T^*\tilde{M}_\beta$, and the integral of $b$ along a loop is always $0$. 
	 
	Therefore, there is a function $G\in C^\infty(M)$ such that $G\circ\pi\circ i=g$.
	We can therefore define a diffeomorphism:
	\begin{align*}
		s:T^*\tilde{M}'&\rightarrow T^*\tilde{M}'\\
		(q,p)&\mapsto (q,e^{-G(q)}\times p)
	\end{align*}
	
	We can then conclude this case by pointing out that, if the pullback of the canonical Liouville form of $T^*M$ on $L$ is $d_\beta f$, then by construction $(s\circ j')^*\lambda=d(e^{-g}f)$. This means that $s\circ j'(L)$ is a $0$-exact Lagrangian of $T^*\tilde{M}'$ but that the projection on $\tilde{M}'$ is not a homotopy equivalence (by hypothesis). Absurd (this contradicts the Abouzaid-Kragh theorem).\newline
	
	Let us now do the general case where $[\beta]\in H^1(M,\mathbb{R})$. Let us write $[\beta]  =\Sigma_{i=1}^r b_i[\beta_i]$ as in the previous lemma.
	 Note that , $i^*[\beta]=0$ if and only if for every $i$, $ i^*[\beta_i]=0$. therefore, by repeating the above construction for each $\beta_i$, we can conclude that, there is some $T^*\tilde{M}'$ a finite cover of $T^*M$ in which a $0$-exact Lagrangian submanifold is not homotopically equivalent to $\tilde{M}'$ through the projection, which is absurd.
\end{proof}

The theorem \ref{thm2} has a few consequences for $\lcs$ geometry, when coupled with Abouzaid-Kragh's theorem:
\begin{Cor}
Under the same conditions as theorem \ref{thm2}, the map $H_1(L)\rightarrow H_1(M)$ induced by the projection is never trivial if $M$ has non-zero first Betti number.
\end{Cor}
A consequence of which is:

\begin{Cor}
Under the same conditions as theorem \ref{thm2}, a closed exact Lagrangian in a cotangent bundle is never contractible.
\end{Cor}
Moreover, the theorem also yields the following corollary:
\begin{Cor}\label{amusant}
Given $M$ and $L$ some connected closed manifolds, let $\beta\in\Omega^1(M)$ be closed and $i:L\rightarrow T^*M$ be a Lagrangian embedding such that  $i^*\lambda=d_\beta f$. If $d_\beta f= d_\beta g$ for some map $g\in C^\infty(L)$, Then $f=g$.
\end{Cor}

\begin{proof}
Let $h\in C^\infty(L)$ be such that $d_{i^*\beta}h=dh-hi^*\beta=0$. Therefore, away from $h=0$, we have that $d\ln(|h|)=i^*\beta$. If $h$ is non-zero on some open set $V$, then $d\ln(|h|)$ goes to infinity when near to $\partial V$, and therefore, either $V=L$ (absurd since $i^*\beta$ cannot be exact by the theorem), either $V=\varnothing$. Therefore, $h$ is always zero.
\end{proof}

Finally, let us state a last corollary, about the dynamics of exact Lagrangians:

\begin{Cor}
	Let $M$ be a connected closed manifold, and $L$ be a $\beta$-exact Lagrangian that is a connected closed submanifold of $T^*M$ ($\beta$ the pullback of some closed non-exact $1$-form on $M$). If there is a map $H\in C^\infty(T^*M)$ such that $TL\subset\ker(d_\beta H)$, then $L\subset\{H=0\}$.
\end{Cor}

This corollary stems directly from putting together theorem \ref{thm2} of this paper and the corollary 4 in \cite{Allais2024TDCHFDC}.

\section{An extension theorem}\label{extension}

Let $M$ be a closed manifold, $\beta\in\Omega^1(M)$ be closed, and $\lambda$ be the canonical Liouville form on $T^*M$. Let $L$ be a $\beta$-exact submanifold of $(T^*M,\lambda,\beta)$ such that the pullback of $\lambda$ to $L$ is equal to $d_\beta f$ for some map $f\in C^\infty(L)$.\newline

Let us note that had the equation been $i^*\lambda=df-\beta$, then a simple translation by $\beta$ would have made $L$ a $0$-exact Lagrangian submanifold. Therefore, provided that $f>0$ and that it can be extended to some smooth map $F$ on $T^*M$ such that $\frac{\lambda}{F}$ is a Liouville form, we can translate $L$ to be a $0$-exact submanifold of $(T^*M,\frac{\lambda}{F})$. As we will later see, we would then be able to apply Moser's trick to find a isotopy between the translation of $L$ and a $0$-exact Lagrangian submanifold of $(T^*M,\lambda)$.

This encourages us to consider the slightly more general case where we have some $g:T^*M\rightarrow\mathbb{R}$ such that:
\begin{align*}
	1.\;& \frac{1}{g}\lambda\textrm{ is a Liouville form}\\
	2.\;&i^*\left(\frac{1}{g}\lambda\right)=dh-\eta\textrm{ for } h\in C^\infty(L)\textrm{ and } \eta\in\Omega^1(L)\textrm{ closed}
\end{align*}

\paragraph{\textbf{Analysis, first part}}\label{Ana1} A quick computation yields the following equality:
\begin{align}\label{analysisfirstpart}
d\left(\frac{1}{g}\lambda\right)^{\wedge n}=&\left(\frac{1}{g}\right)^n\omega^{\wedge n}+n\left(\left(d\frac{1}{g}\wedge\lambda\right)\wedge\left(\left(\frac{1}{g}\right)^{n-1}\omega^{\wedge n-1}\right)\right)
\end{align}

We can note here that considering $\frac{1}{g}\lambda$ instead of $G\lambda$, for some map $G\in C^\infty(M)$, does not lead to a loss of generality. Indeed, if $G(x)=0$, then $d(G\lambda)_x= (dG \wedge\lambda)_x + G(x)d\lambda_x = (dG \wedge\lambda)_x$. Moreover, we will assume, without loss of generality, that $g>0$.

The first equality of this analysis (equality \ref{analysisfirstpart}) implies that $d(\frac{1}{g}\lambda)$ is non-degenerate if and only if, in local coordinates, $(q,p)\in T^*_qM$:
\[\frac{n!}{g}\times\bigwedge_j dp_j\wedge dq_j\neq -n! \sum_i p_i\partial_{p_i}\left(\frac{1}{g}\right)\bigwedge_j dp_j\wedge dq_j\]
This condition can be summed up by the (in)equality:
\[1\neq g\times\frac{dg}{g^2}(Z_\lambda)=d\ln(g)(Z_\lambda)\]
where $Z_\lambda$ is the Liouville vector field of $\lambda$. Since $d\ln(g)(Z_\lambda)_{|(q,0)}=0$, we get the inequality:
\[d\ln(g)(Z_\lambda)<1\]

\paragraph{\textbf{Analysis, second part}}\label{Ana2} Given a map $h:L\rightarrow\mathbb{R}$ such that $\frac{df}{h}-\frac{f}{h}i^*\beta$, we want to understand the conditions under which this map can be extended to a map $g:T^*M\rightarrow\mathbb{R}$ fulfilling the conditions laid out in the first part. It is immediately apparent that the mean value theorem can obstruct such a map. The mean value theorem does not offer an obstruction if and only if (in local coordinates): 
\[\frac{ln(h(q,tp))-ln(h(q,p))}{ln(t)}<1\]
where $(q,p),(q,tp)\in T^*_qM\cap i(L)$, for some $t\in\mathbb{R}_+^*-\{1\}$. For $t=0$, this inequality is always true.

We will say that the MVT obstructs the extension if the inequality does not hold for at least a pair of such points. We will endeavor to show that it is the only obstruction.\newline\newline

Henceforth, we will assume that our exact Lagrangians are ``generic enough'' in the following sense:
\begin{Hyp}\label{HypLagHolo} Let $M$, $L$ and $i:L\rightarrow T^*M$ be as previously defined. We will assume that $i(L)$ fulfills the following properties:
	\begin{itemize}
		\item[1)] \textit{$i(L)$ is transverse to the $0$-section $M$.} Call $\{q_1,\ldots,q_k\}$ the intersection points.
		\item[2)] \textit{The projections of $Di(TL)$ and $VT^*M$ in $\mathbb{P}T^*M$ only have transverse intersections}, with $VT^*M$ the vertical bundle of $T^*M$. We will call $I$ the immersed submanifold given by the projection of the intersection on $T^*M$.
		\item[3)] \textit{$I$ and the $0$-section $M$ do not intersect.}
	\end{itemize}
\end{Hyp}
Let us elucidate a bit the meaning of those hypothesis. In particular, let us focus on the second hypothesis. Let $x$ be a point of $L$ such that $L$ has at least one ``direction'' that is tangent to the vertical bundle ($Di(T_xL)$ and $V_xT^*M$ intersect along some vector subspace of dimension at least $1$). Since the intersection is always a vector subspace, let us get rid of the superfluous information by taking the projectivization. Here, the projection of $Di(TL)$ in $\mathbb{P}T(T^*M)$ is of dimension $n+(n-1)=2n-1$ and the projection of $VT^*M$ in $\mathbb{P}T(T^*M)$ is of dimension $2n+n-1=3n-1$, meanwhile the total space is of dimension $4n-1$. Therefore, Thom's jet transversality theorem implies that the intersection is generically a $(n-1)$-dimensional manifold. Therefore, 
$I=\{x\in L:T_xL\textit{ has a tangent direction to }VT^*M\}$. We will now detail how to use Thom's jet transversality to get this result. This will lead us to prove the following :

\begin{Lem}
The second hypothesis holds true for a generic submanifold of dimension $n$ of $T^*M$
\end{Lem}

The jet space  $J^1(L,T^*M)$ is, locally, the space $L\times T^*M\times\mathcal{M}(n,2n)$ where $\mathcal{M}(n,2n)$ is the space of matrices of $\mathbb{R}^n$ to $\mathbb{R}^{2n}$. Notice that we can split $\mathbb{R}^{2n}$ as $\mathbb{R}_1^n\oplus\mathbb{R}^n_2$ where $\mathbb{R}_1^n$ corresponds to the horizontal bundle whereas $\mathbb{R}_2^n$ corresponds to the vertical bundle. Now, let us consider $W$, the submanifold of $J^1(L,T^*M)$ which can be locally written as $L\times T^*M\times\mathcal{M}_1(n,2n)$ where $\mathcal{M}_1(n,2n)$ is the subspace of matrices of $\mathcal{M}(n,2n)$ such that the projection of the image on $\mathbb{R}_1^n$ isn't surjective (it's the subspace of matrices $\big(A\;\;B\big)$ with $A,B\in\mathcal{M}(n,n)$ and $rank(A)<n$). However, the space of matrices of $\mathcal{M}(n,n)$ of rank $r$ is a  $rn+(n-r)r$-dimensional submanifold. Therefore, $W$ is a (collection of) submanifold(s) of dimension at most $n-1$. Since $Di$ is of maximal rank (equal to $dim(L)$) and $L$ is a compact set, any small enough perturbation of $i$ will also be of maximal rank. Moreover, if $j^1 i$ intersects $W$ at $j^1i_l$ for some $l\in L$, then there is $X\in T_l L$ such that $Di_l(X)\in V_{i(l)}T^*M$ since, to reiterate, $Di_l$ is of maximal rank. Therefore, the intersection of the two projections in $\mathbb{P}T(T^*M)$ is generically z submanifold of dimension $n-1$ called $\tilde{I}$.

We can now state the following theorem, that will be integral to the proof of theorem \ref{thm3}.
\begin{The}\label{ThmTAF}
	Let $M$ be a connected closed manifold, $\lambda$ be the canonical Liouville form on $T^*M$ and $L$ be a closed submanifold of $T^*M$ satisfying the genericity conditions stated above. Let $V$ be a neighborhood of  $L$.
	
	Then, for every map $h:V\mapsto\mathbb{R}$ such that $h_{|L}$ is not MVT obstructed and such that \[d\ln(h)(Z_\lambda)<1,\] there is a map $g:T^*M\mapsto\mathbb{R}$ equal to $h$ on a neighborhood of $L$ and equal to $1$ outside of a compact, such that $g$ satisfies: 
	\[d\ln(g)(Z_\lambda)<1\]
\end{The}

The proof of the theorem \ref{ThmTAF} will take up the next pages and will be followed by a few corollaries. The proof will essentially be in two parts. First, a lemma which will allow us to restrict ourselves to only having to extend $h$ in the neighborhood of a good $n+1$-dimensional ``polyhedron'' in $T^*M$, following the same philosophy as the $h$-principle for open manifolds. Then, we will extend the application $h$ (as describe in the theorem) to a neighborhood of said ``polyhedron''.

\section{Proof of the theorem \ref{ThmTAF}}\label{proof}

\subsection{Restriction to neighborhoods of a good polygon} 

We will first use the classic trick (see, for example, \cite{Eliashberg2002IntroductionTT}) of restricting ourselves to the neighborhoods of a polyhedron. Let us start by defining the core of a manifold.
\begin{Def}
	 Let $V$ be a manifold, $K\in V$ is called a core if $K$ is a polyhedron of positive codimension and, for any neighborhood $U$ of $K$, there is an isotopy $\phi_t:V\rightarrow V$ such that $\phi_{t|K}=id_{|K}$ and $\phi_1(V)\subset U$.
\end{Def} 

\begin{Rem}
Note that this is the definition used when proving the $h$-principle for open manifolds. Here, however, we do not need $K$ to be a polyhedron. We will however keep using this terminology even though, here, $K$ is not required to be a polyhedron
\end{Rem}

Choose some arbitrary Riemannian metric $<\cdot,\cdot>$ on $T^*M$ and define $\mathbb{B}_r^* M$, the ball sub-bundle of radius $r$ of $T^*M$. Assume that $r$ is large enough that $i(L)\subset\mathbb{B}_r^* M$.

\begin{Def}
	Let $K$ be the set:
	\[K:=\{(q,tp)\in T^*M\,:\,(q,p)\in i(L),0\leq t\leq1\}\cup M\]
	
	For each $q\in M$, we will call $K\cap T_q^*M$ the star (above $q$).
	
	For each $(q,p)\in i(L)$, $\{(q,tp)\,:\,0\leq t\leq 1\}$ will be the branch of the star.
\end{Def}
\begin{Lem}
	For $K$ defined as above, $K$ is a core of the manifold $\mathbb{B}_r^* M$.
\end{Lem}
\begin{proof} Let $U$ be a neighborhood of $K$.
	Take a cover of $\mathbb{B}_r^*M$ by open sets $(\Omega_i)_{i\in I}$ small enough so that $\cup_{\{i\,:\,\Omega_i\cap K\neq\emptyset\}}\Omega_i\subset U$. Take $(\psi_i)_{i\in I}$ a smooth partition of unity subordinate to $(\Omega_i)_{i\in I}$ and define: \[d:(q,p)\in T^*M\mapsto \sum_{\{i\,:\,(q,p)\in\Omega_i\, ,\,\Omega_i\cap K=\emptyset\}}\psi_i(q,p)\in \mathbb{R}\]
	This map is smooth, is equal to $0$ near $K$ and $1$ outside of some arbitrarily small neighborhood of  $K$.
	
	We can now define the vector field: 
	\[\tilde{Z_\lambda}:(q,p)\in T^*M\mapsto -d(q,p)\times (Z_\lambda)_{(q,p)}\in TT^*M\]
	The flow $\phi^t_{\tilde{Z_\lambda}}$ of this vector field yields the desired isotopy. Indeed, let us take $V$ an arbitrarily small closed neighborhood of the $0$-section such that its convex envelope (in $T_q^*M$, for each $q\in M$) with $K$ is a subset of $U$. Let us call this envelope $C$. We can assume that $\cup_{\{i\,:\,\Omega_i\cap K=\emptyset\}}\Omega_i\subset T^*M\backslash C$.
	
	 Notice that each flow line of $-Z_\lambda$ intersects $C$ exactly once. Therefore, $\phi^t_{\tilde{Z_\lambda}}$ takes each point of $\mathbb{B}_r^*M$ to an arbitrarily small neighborhood of $K$ in finite time. Also note that the flow can only get closer to $C$ as $t$ increases. Since $\mathbb{B}_r^*M$ is compact, there is a global maximum for those times. Therefore, up to multiplying $\tilde{Z_\lambda}$ by a positive factor, we can assume that the global maximum for those times is at most $1$.
\end{proof}

However, the pullback of a map $g$ satisfying $d\ln(|g|)(Z_\lambda)<1$ will not, in general, also satisfy the inequality. Therefore, it behooves us to be a bit more careful with our pullbacks.
\begin{Pro}\label{ProNghd}With the same notations as above:
	\begin{itemize}
		\item There is a decreasing basis of neighborhoods, $U_i\supset U_{i+1}$, of $K$ such that for every $q\in M$ and every $(q,p)\in T^*_qM-\{(q,0)\}$, $\{(q,tp)\,:\,t\in\mathbb{R}$ and $\partial (U_{i}\cap T_q^*M)$ are transverse and intersect exactly either twice or zero times.
		\item For each $U_i$, there is a diffeomorphism $\phi_i$ of $\mathbb{B}_r^*M$ in $U_i$ of class $C^1$ such that $\phi_{i|K}=id$ and such that if $d\ln(g)(Z_\lambda)<1$ for some $g: U_i\rightarrow \mathbb{R}$, then $\phi^*d\ln(g)(Z_\lambda)<1$.
	\end{itemize}
\end{Pro}

The proof of this proposition is split in three. First, we will prove the first point of the proposition, then we will use an analysis/synthesis to prove the second point.
\begin{proof}
	\textit{First part}:
	
	Take $r'>0$ to be the smallest injectivity radius of the exponential map over a point of $K$. For each point $(q,p)\in K$ and each $i>0$, we consider the open set \[V^i_{(q,p)}=\bigg\{x\in T^*M\,:\,dist(x,(q,p))<\frac{D}{\|(q,p)\|+i}r'\bigg\}\]
	where $D$ is the length of the branch on which lays $(q,p)$, and $dist$ is the distance induced by the metric.
	Take $U_i=\cup_{(q,p)\in K}V^i_{(q,p)}$. Those open sets fulfill the conditions.\newline
	
	\textit{Second part, Analysis}:
	
	Let $X$ be an element of $\mathbb{S}^*M$ where $\mathbb{S}^*M$is the sphere sub-bundle of $T^*M$ defined by the metric. Using the proof of the preceding lemma for inspiration, assume that the diffeomorphism of the second point can be written as $\phi(tX)=\alpha_{X}(t)\times tX$ for some positive map $\alpha:\mathbb{S}^*M\times\mathbb{R}_+\rightarrow\mathbb{R}_+$.
	
	Take $(q,p)\in T^*M$ such that $q\neq0$. Define $v=(q,\frac{p}{\|p\|})$, $t'=\|p\|$ and $w=\frac{Z_\lambda}{\|Z\lambda\|} $. Notice that $(q,p)=t'v$ and $Z_\lambda=t'w$. Let $g$ be a map defined on $U_i$ such that $d\ln(g)(Z_\lambda)<1$. To clarify the computations, take $s=\ln(g)$ and $S=d\ln(g)$, then for $t=t'$, we have the following computation:
	\[\phi_i^*S_{tv}=  S_{\phi_i(tv)}(\alpha_v(t)\partial_w\otimes d_w)+S(t\alpha_v'(t)\partial_w\otimes d_w)\]
	However, by hypothesis on $\sigma=(s,S)$, we have that \[1> S_{\phi_i(tv)}\big((\alpha_v(t)\partial_w\otimes d_w)_{tv}(Z_{\lambda |tv})\big)\] with $(\alpha_v(t)\partial_w\otimes d_w)(Z_{\lambda |tv}) \in T^*_{(q,\alpha_v(t)tv)}T^*M$. The parallel transport of $Z_{\lambda|tv}$ to $\phi_i(tv)$ yields the following equality $\alpha_v(t) Z_{\lambda|tv}=Z_{\lambda|\alpha_v(t)tv}$.
	
	To simplify computations,assume that $\alpha$ is $C^1$-close to the map $(\frac{r}{r_0+\epsilon})^{-h(t)}$ with $r_0=dist(0,\partial U_{i+k}\cap Vect^+(v))$ for some $k$ big enough and $\epsilon$ small enough, such that $(r_0+\epsilon) v\in U_i$. The map $h:\left[r_0,r\right]\rightarrow \left[0,1\right]$ is defined such that $h$ has a positive derivative. We will now focus on the study of $h$ to find out what properties this map must fulfill.\newline
	
	According to our hypothesis on $\alpha$, $S(t\alpha_v'(t)\partial_w\otimes d_w)$ is $C^0$-close to  \[S\bigg(-t\ln\bigg(\frac{r}{r_0+\epsilon}\bigg)h'(t)\alpha_v(t)\partial_w\otimes d_w\bigg).\]
	This yields:
	\[\phi_i^*S_{tv}(Z_\lambda)=S_{\phi_i(tv)}\left[(\alpha_v(t)\partial_w\otimes d_w+t\alpha_v'(t)\partial_w\otimes d_w)_{tv}(Z_\lambda)\right]\] \[\approx\left[1-th'(t) ln(\frac{r}{r_0+\epsilon})\right]\times S_{\phi_i(tv)}\left[(\alpha_v(t)\partial_w\otimes d_w)_{tv}(Z_\lambda)\right]\]\[ <\left[1-th'(t) ln(\frac{r}{r_0+\epsilon})\right]\]
	Therefore, $h$ must satisfy $1>1-th'(t)ln(\frac{r}{r_0+\epsilon})>0$, which implies that $0<h'(t)<1/(tln(\frac{r}{r_0+\epsilon}))$.\newline
	
	\textit{Synthesis:}
	
	The following choice of $h$ satisfies the conditions laid out in the analysis: 
	\[h(t)=\frac{\ln\left(\frac{t}{r_0}\right)+\left( \frac{t-r_0}{r-r_0}\right)\ln\left(\frac{r_0}{r_0+\epsilon}\right)}{\ln\left(\frac{r}{r_0+\epsilon}\right)}\]
	
	Indeed, for some $t'=t\in \left[r_0,r\right]$, we have $ln(1+\epsilon/r_0)<(r-r_0)/r$ for some $\epsilon$ small enough.  With this choice, \[\phi_i(tv)=r_0 (\frac{r_0+\epsilon}{r_0})^\frac{t-r_0}{r-r_0}v\] and \[\partial_t\phi_i(tv)=\frac{r_0}{r-r_0}\ln\left(\frac{r_0+\epsilon}{r_0}\right)\left(\frac{r_0+\epsilon}{r_0}\right)^\frac{t-r_0}{r-r_0}v\]
	
	For $0\leq t<r_0-\epsilon$, let $\phi_i(tv)=tv$.\newline
	
	For $r_0-\epsilon\leq t<r_0$, we will take $\phi_i$ an interpolation such that $\phi_i((r_0-\epsilon)v)=(r_0-\epsilon)v$ and $\partial_t\phi_i((r_0-\epsilon)v)=(r_0-\epsilon)v$, with $\phi_i(r_0v)=r_0v$ and $\partial_t\phi_i(r_0v)=\frac{r_0}{r-r_0}\ln\left(\frac{r_0+\epsilon}{r_0}\right)v$. We can also pick $\phi_i$ such that its directional derivative in the direction $\partial_t$ is at most $1+\epsilon'$ with $\epsilon'$ arbitrarily near  $0$.
	
	For example, we can take $\phi_i(tv)=\alpha_v(t)v$ with:
	\begin{align*}
		\alpha_v(t)=H\left(\frac{t-r_0+\epsilon}{\epsilon}\right) & \left(r_0 +(t-r_0)\frac{r_0^2}{r-r_0}\ln\left(\frac{r_0+\epsilon}{r_0}\right)\right) \\
		& +\left(1-H\left(\frac{t-r_0+\epsilon}{\epsilon}\right)\right)t
	\end{align*}
	where $H$ is a primitive of $t\in(0,1)\mapsto b e^{-\frac{a}{t}+\frac{a}{t-1}}$, extended at $0$ and $1$ by continuity and chosen such that $H(0)=0$, with $a\in\mathbb{R}_+^*$ and $b$ chosen such that $H(1)=1$.
	
The derivative (along $t$) of the map $\alpha_v$ is:
	\begin{align*}
		H\left(\frac{t-r_0+\epsilon}{\epsilon}\right) & \frac{r_0^2}{r-r_0}\ln\left(\frac{r_0+\epsilon}{r_0}\right)+ \left(1-H\left(\frac{t-r_0+\epsilon}{\epsilon}\right)\right)\\
		& +(\partial_t H)\left(\frac{t-r_0+\epsilon}{\epsilon}\right)\frac{t-r_0}{\epsilon}\left(\frac{r_0^2}{r-r_0}\ln\left(\frac{r_0+\epsilon}{r_0}\right)-1\right)
	\end{align*}
	The first line of this equation fulfills the conditions for the derivative of $\phi_i$.
	
The second line of this sum is equal to \[b e^{-\frac{a\epsilon}{t-r_0+\epsilon}+\frac{a\epsilon}{t-r_0}}\frac{r_0-t}{\epsilon}\left(1-\frac{r_0^2}{r-r_0}\ln\left(\frac{r_0+\epsilon}{r_0}\right)\right).\] This map reaches its maximum at \[t=r_0-\epsilon+\frac{a\epsilon^3}{2a\epsilon^2+1}.\] The map is therefore bounded from above by \[be^{-\frac{a}{a\epsilon^2+1}-\frac{1}{\epsilon^2}}\times\epsilon\left(\frac{1-a\epsilon^2}{1+2a\epsilon^2} \right)\times\left(1-\frac{r_0^2}{r-r_0}\ln\left(\frac{r_0+\epsilon}{r_0}\right)\right),\] which goes to $0$ when $\epsilon\rightarrow 0$. Therefore, for any $\epsilon$ small enough, the interpolation fulfills the desired conditions $\epsilon$.\newline

	Finally, $r_0,\epsilon:\mathbb{S}^*M\rightarrow\mathbb{R}_+^*$ are continuous maps, that we can replace with smooth approximations without changing the proof.\newline
	
	This gives a $\phi_i$ of class $C^1$ which fulfills the second point of the proposition.
\end{proof}$ $

\subsection{A series of extensions} 
We will finalize the proof of theorem \ref{ThmTAF}. In the previous section, we have seen that we can restrict the problem to a neighborhood of $K$. Let us then show that we can define a map $g$ on a neighborhood of $K$ such that $d \ln(g)(Z_\lambda) < 1$.\newline

\paragraph{\textbf{Assertion 1:}} Any map $h$ defined as in the theorem \ref{ThmTAF} can be extended to a neighborhood of the $0$ section such that the extension is not obstructed by the MVT.\newline

Indeed, away from the intersection points between $L$ and $M$, it is sufficient to take the extension of $h$ to be equal to $\max(h)$. This definition can be extended to a neighborhood of $M$, away from the intersections with $L$. This map can be extended to a constant map on a neighborhood $M$, away from the intersections with $L$. Around the intersection points, one can always take a $C^1$ interpolation between $\max(h)$ (defined on the neighborhood of $M$ away from $L$) and $h$ (defined on a neighborhood of $L$). Call the extension $H$ (defined on a neighborhood of $L\cup M$).

Finally, up to taking a small perturbation away from $L$, $H$ can be smooth.\newline

\paragraph{\textbf{Assertion 2:}} the extension $H$ of $h$ of the previous assertion can be extended to a map $G$ defined in a neighborhood of $K$ such that: \[dln(G)(Z_\lambda)<1.\]
$ $

Take some Riemannian metric on $T^*M$. For any point $q\in M$, consider $T^*_qM\cup U$ where $U$ is the union of a tubular neighborhood of $M$ with smooth boundary and the union of a tubular neighborhood of $L$ with smooth boundary. Then the orbits of $Z_\lambda$ from $\partial U$ to $\partial U$ form a collection of segments. For each such segment, call $l$ (resp. $l'$) the maximum distance (resp. minimum) of a point on this segment to the $0$-section. Finally, note that the union of $U$ and those segments form a neighborhood of $K$.

For each $q\in M$, we will restrict ourselves to arbitrarily small neighborhoods $V_q\subset M$ of $q$ such that $T^*V_q$ has a system of local coordinates defining an isometry with some open subset of $\R^{2n}$. On each aforementioned neighborhood, take the interpolation \[\tilde{G}:t\in[l',l]\mapsto  \exp \bigg(\frac{\ln(l/t)}{\ln(l/l')}\ln(H(q,l'p))+\frac{\ln(t/l')}{ln(l/l')}\ln(H(q,lp)\bigg) \] with $\|p\|=1$ for $(q,l'p))$ and $(q,lp)$ in the domain of definition of $H$. Gluing this interpolation to $H$ along $\partial U$, we get a map that is piece-wise $C^1$. This new map will also be called $\tilde{G}$. Let $\chi_n$ be a positive approximation of unity with arbitrarily small compact support. Then $\chi_n*\tilde{G}$ is smooth and, around each point $(q',p')\in T^*V_q$ there is an arbitrarily small neighborhood $Neig$ 
\[\|Z_\lambda\|\partial_wln(\chi_n*\tilde{G})\leq \frac{\sup_{Neig}(d\tilde{G}(Z_\lambda))}{\inf_{Neig}(\tilde{G})}<1\]
whenever the support is small enough (recall that $w=\frac{Z_\lambda}{\|Z_\lambda\|}$). Indeed, \[\sup_{Neig}(d\tilde{G}(Z_\lambda))\leq \max\left(\sup_{Neig}\left(\frac{\ln(H(q,lp))-\ln(H(q,l'p))}{\ln(l/l')}\right)\times \sup_{Neig}(\tilde{G}),\sup_{Neig}(dH(Z_\lambda))\right).\] 

Remark that $\chi_n*\tilde{G}_{|L\cup M}\neq h_{|L\cup M}$. We will then do a new interpolation to a neighborhood of $M\cup L$. To do this, it is enough to take an interpolation between the Taylor extensions of $H$ and $\chi_n*\tilde{G}$. Since those two maps are $C^1$-close, the interpolation is also going to be $C^1$-close to those maps.

Now, recall that we are in the cotangent of some neighborhood $V_q\subset M$ of some $q\in M$. Since $M$ is assumed to be closed, we can take a finite sub-cover of $(V_q)_{q\in M}$ and a partition of unity subordinate to this cover. This partition can be extended along the cover $(T^*V_q)_{q\in M}$ of $T^*M$ such that the extensions of the partitions are constant along the fibers (of $T^*V_q$). Using this partition of unity and the various (locally defined) $\chi_n*\tilde{G}$, we can globally define a map $G$ satisfying the conclusions of the theorem \ref{ThmTAF} in a neighborhood of $K$.\newline

\paragraph{\textbf{Assertion 3:}} The map $G$ can be extended a map $g$ on $T^*M$ in such a way that $g$ is constant equal to $1$ outside of a compact and \[d\ln(g)(Z_\lambda)<1.\]$ $

It is indeed enough to combine the previous assertion with the previous proposition. This gives an extension of $G$ to an arbitrarily big ball bundle (say of radius $r$). Taking $G'$ the constant map equal to $1$ defined outside of a ball bundle of radius $r'$ sufficiently large (compared to $r$), then any interpolation with small enough slope will do.

It is indeed enough to combine the second assertion and the proposition \ref{ProNghd}. This defines an extension $G'$ on some arbitrarily big ball bundle $\mathbb{B}_r^*M$. One can then take some  $R$ that is arbitrarily large when compared to $r$, and take a smooth interpolation between our map (defined on $\mathbb{B}_r^*M$) and the constant map equal to $\inf(G)$ on $T^*M\backslash\mathbb{B}_R^*M$. One can then repeat the process on a bigger ball bundle to interpolate our map and the constant map equal to $1$. This final map will verify the conclusion of the theorem  \ref{ThmTAF}.

\section{On the projection of exact Lagrangians}\label{projection}
\subsection{Some corollaries and proof of theorem \ref{thm3}}
The previous theorem has a couple of corollaries. The theorem \ref{thm3} will appear as one of them.

\begin{Cor}\label{CorTAF1}
	Let $M$ and $L$ be connected closed manifolds, $\lambda$ be the canonical Liouville form on $T^*M$ and $i:L\rightarrow T^*M$ be an exact Lagrangian submanifold, satisfying the conditions of the theorem \ref{ThmTAF}, such that $i^*\lambda=df+fi^*\beta$  for some $f\in C^\infty(L)$.  Let $V$ be a neighborhood of $i(L)$.
	
	Assume that there is a map $h\in C^\infty(V,\mathbb{R}_+^*)$, satisfying the conditions of \ref{ThmTAF}, and such that: \[\frac{df}{i^*h}+\frac{f}{i^*h}i^*\beta=dH-\eta\] for $H$ a smooth function on $L$ and $\eta\in\Omega^1(L)$ closed.
	
	Then there is an isotopy $\phi:T^*M\times[0,1]\rightarrow T^*M$ such that $\phi_0=id$, $\phi_1(i(L))$ is a Lagrangian of $(T^*M,\lambda)$ and for each $q\in M$ and each $t\in[0,1]$, $\phi_t(T^*_qM)=T^*_qM$.
\end{Cor}

\begin{proof} 
	Note that if $g_0\lambda$ is a Liouville form, then $\lambda_t=(tg_0+1-t)\lambda$ is a Liouville form for every $t\in[0,1]$. Indeed, as stated earlier $\lambda_t$ is a Liouville form if and only if $-D\ln(tg_0+1-t)(Z_\lambda)<1$, that is to say that \[\frac{tDg_0(Z_\lambda)}{tg+1-t}>-1.\] Notice that \[\partial_t \frac{tDg_0(Z_\lambda)}{tg_0+1-t}=\frac{Dg_0(Z_\lambda)}{(tg_0+1-t)^2},\] the sign of which does not depend on $t$. Therefore, \[\left(\frac{tDg_0(Z_\lambda)}{tg_0+1-t}\right)_{|x}\geq\min(D\ln(g_0)(Z_\lambda),0)\geq -1\] since $g_0\lambda$ is a Liouville form (and as such $D\ln(g_0)(Z_\lambda)\geq -1$).

This allows us to prove the following lemma:
\begin{Lem}\label{Mosertrick}
	Let $\omega$ be the canonical symplectic form on $T^*M$. If $\omega'=d\frac{\lambda}{g}$ is a symplectic form, for some map $g$ equal to $1$ outside of a compact, then there is an isotopy $\phi:(T^*M,\omega')\times[0,1]\rightarrow (T^*M,\omega)$ such that $\phi_0=id$ and $\phi_1$ is a Liouville diffeomorphism, and such that for each $q\in M$, the restriction of $\phi_t$ on $T_q^*M$ is a diffeomorphism of $T^*_qM$ into itself.
\end{Lem}
\begin{proof}
	Let $g_t$ be a time-dependent map from $T^*M$ to $\mathbb{R}$. Assume that for each $t$, $d g_t\lambda$ is symplectic, Then we can use Moser's trick for $dg_t\lambda=:\omega_t$.
	
	Let us find a time-dependent vector field $X_t$ such that $\partial_t\omega_t+\mathcal{L}_{X_t}\omega_t=0$. Any vector field satisfying that equality must also satisfy:
	\begin{align*}
		0 & = \partial_t\omega_t+\mathcal{L}_{X_t}\omega_t=\partial_tdg_t\lambda+d\omega_t(X_t,\cdot)=d(\partial_tg_t\lambda+\omega_t(X_t,\cdot))
	\end{align*}
	Therefore, it is enough for $X_t$ to satisfy $\omega_t(X_t,\cdot)=\partial_t g_t\lambda$. This equation always has a solution since $\omega_t$ is a family of non-degenerate $2$-form.
	
Note that $g_t\lambda$ is indeed the primitive of a symplectic form as soon as $D\ln (1/g_t)(Z_\lambda)<1$, which is to say $D\ln (g_t)(Z_\lambda)>-1$. Let $Y_t$ be a solution of $\omega_t(Y_t,\cdot)=\lambda$. Notice that this is equivalent to \[g_t\omega(Y_t,\cdot)=\lambda-\iota_{Y_t}dg_t\wedge\lambda.\] Assuming that $Y_t$ is colinear to $Z_\lambda$ (i.e. $Y_t=\alpha_t Z_\lambda$ for some map $\alpha_t$), we have:
	\begin{align*}
		\alpha_tg_t\lambda & =\lambda-\alpha_tdg_t(Z_\lambda)\lambda\\
		\implies \alpha_t&=\frac{1}{g_t+Dg_t(Z_\lambda)}
	\end{align*}
Notice that we always have $Dg_t(Z_\lambda)>-g_t$, and therefore $\alpha_t$ is always well-defined.
This implies that \[X_t=\alpha_t\partial_tg_t Z_\lambda=\frac{\partial_tg_t}{g_t+Dg_t(Z_\lambda)}Z_\lambda.\]

If $g_t$ is equal to $1$ outside of a compact, then the flow of $X_t$ is well-defined at each time.

We can therefore take $g_0=\frac{1}{g}$ and  $g_t=tg_0+(1-t)$.
\end{proof}
	
This lemma shows that there is an isotopy $\phi$ preserving the fibers of $T^*M$ and such that $\phi_1^*\lambda=\frac{\lambda}{h}$. Therefore \[d((\phi_1\circ i)^*\lambda)=d(dH+d\tilde{H}-\eta)=0,\] for a smooth map $\tilde{H}$ on $L$. 
\end{proof}

As a consequence of the previous corollary, we have:
\begin{Cor}\label{CorTAF2}
	Under the same conditions as corollary \ref{CorTAF1}, assume that  $\eta= i^*\tilde{\eta}$ for some closed $\tilde{\eta}\in\Omega^1(T^*M)$. Let $\pi:T^*M\rightarrow M$ be the projection.
	
	Then $\pi$ induces a simple homotopy equivalence between $i(L)$ and $M$. 
\end{Cor}
\begin{proof}
	Observe that for any $\eta'\in\Omega^1(M)$, we have a fiber-preserving isotopy of $T^*M$ given by $((q,p),t)\mapsto (q,p+t\eta')$. for $t=1$, the corresponding diffeomorphism will be called the translation by $\eta'$.
	
	Taking $\eta'$ in the same homology class as $\eta$, the composition of the isotopy given in the proof of the corollary \ref{CorTAF1} with the isotopy given above gives a fiber-preserving isotopy of $T^*M$ sending $L$ on a $0$-exact Lagrangian. We then combine this observation with Abouzaid-Kragh's theorem to get the result
\end{proof}

This last corollary is almost enough to prove the theorem \ref{thm3}. Indeed, it is sufficient to show that under the conditions of the theorem, there is an extension $F$ of $f$ to a neighborhood of $i(L)$ such that $F$ fulfills the conditions of the theorem \ref{ThmTAF}

\begin{Lem}
 Under the same conditions as \ref{thm3}, assume that $f$ is positive then there is an extension $h$ of $f$ to a neighborhood of $i(L)$ such that $h$ satisfies the conditions of the theorem \ref{ThmTAF}.
\end{Lem}

\begin{Rem}
	Among other things, we will prove that 	if $(Z_\lambda)_x\in T_xi(L)$, then $d\ln(f)(Z_\lambda)_x=0$.
\end{Rem}

\begin{proof}
By the tubular neighborhood theorem, there is a neighborhood $V_1$ of $L$ in $TL$ which is diffeomorphic to a neighborhood of $i(L)$ in $T^*M$. Call $\phi$ this diffeomorphism. We can therefore pull back $Z_\lambda$ to define a vector field $D\phi^{-1}(Z_\lambda)$ on $V_1$. Since 
\begin{align*}
d_\beta f(Di^{-1}(Z_\lambda))&=df(Di^{-1}(Z_\lambda))-fi^*\beta(Di^{-1}(Z_\lambda))\\ &=df(Di^{-1}(Z_\lambda))=i^*\lambda(Di^{-1}(Z_\lambda))=0,
\end{align*} any extension $h'$ of $f$ to $V_1$ will satisfy \[d\ln(h')(Di^{-1}(Z_\lambda))<1\] sufficiently near $L$, and close to the points where $Z_\lambda\in Di(TL)$.
 
 Let us show that $h'$ can be globally defined on $V$ such that \[d\ln(h')(D\phi^{-1}(Z_\lambda))<1.\] Take $X=D\phi^{-1}(Z_\lambda)$, and call $X_V$ and $X_H$ the vertical and horizontal components of $X$. Take a random metric on $T^*M$ and pull it back to a neighborhood of the $0$-section in $T^*L$. Then for any vector field $V$ of the vertical bundle of $T^*L$, we can define in local coordinates in $T^*L$: \[h'(e^V_{(q,0)})=-df_q(X_{H|(q,0)})\big<\frac{X_V}{\|X_V\|},V\big>_{(q,0)}+f(q)\] with $e$ the exponential map induced by the metric. Note that this map is well-defined when $X_V=0$ since then $df(X_H)=0$. Finally, note that this map satisfies $d\ln(h')(X)_{(q,0)}<1$.
 
 Do be careful that this may not be differentiable at $(q,p)$ when $p\neq 0$ and $Z_\lambda\in Di(T_qL)$, even though the vertical derivative exists and is continuous. Therefore, we can take a simple interpolation between $h$ and the constant map $(q,p)\in T_qL\mapsto f(q)$ in a  neighborhood of $T_qL$ where $Z_\lambda\in Di(T_qL)$, such that the derivative is $C^1$-close to $h'$ along $L$. Let us also name this interpolation $h'$.
 
  Therefore, up to taking a smaller neighborhood $V_1$, \[d\ln(h')(Z_\lambda)\simeq 0<1.\] We then just have to set $h=h'\circ\phi^{-1}$.

\end{proof}

\begin{Rem}
A corollary of this proof is that the length of the essential Liouville chords of a closed exact Lagrangian $L$ (see definition \ref{liouv}) is bounded from below by a non-zero constant. Indeed, we have just shown that, for any $x\in T^*M\cap L$, if $(Z_\lambda)_x$ is tangent to $T_xL$, then Liouville chords are not MTV-obstructed in a neighborhood of $x$.
\end{Rem}

The theorem \ref{thm3} is an immediate consequence of the above lemma and the corollary \ref{CorTAF2}.\newline

Yet another corollary is the following:
\begin{Cor}(with the same notations as in corollary \ref{CorTAF1})
Let $L$ be an exact Lagrangian submanifold fulfilling the conditions of \ref{CorTAF2}. If the projection does not induce a simple homotopy equivalence between $L$ and $M$, then for any Hamiltonian isotopy $\phi_t$, the pullback of $\lambda$ on $\phi_t(L)$ does not allow for a map $h$ satisfying the conditions of \ref{CorTAF1} and \ref{CorTAF2}.
\end{Cor}
As a final reality check, let us remark the following:
\begin{Rem}
Notice that in the case of $0$-exact Lagrangians, we have some freedom in the choice of the primitive of the Liouville form. Using this, we can always find a primitive that is not MVT-obstructed.

However, for $\beta$-exact Lagrangians (for some non-exact $\beta$), we do have that freedom of choice by corollary \ref{amusant}.
\end{Rem}

\subsection{Liouville chords and Reeb chords}

\begin{Def}[(essential) Liouville chord]\label{liouv}
	For $k\in\{1,2\}$, let $L_k$ be exact Lagrangian submanifolds of $(T^*M,\lambda,\beta)$ that intersect transversely such that $i^*\lambda=d_\beta f_k$ for $i$ the inclusion and $f_k\in C^\infty(L_k,\R_+^*)$. We will assume that $\beta$ is not exact. Call $\Phi$ the flow of the canonical Liouville vector field. For any $x\in T^*M$, we have the following definitions:
	\begin{enumerate}
		\item 	A (positive) Liouville chord from $L_1$ to $L_2$ is a trajectory $\{\Phi_s(x):s\in[0,t]\}$ such that $\Phi_0(x)\in L_1$ and $\Phi_t(x)\in L_2$. We will call $t$ the length of the Liouville chord.
		\item A negative Liouville chord from $L_1$ to $L_2$ is a trajectory $\{\Phi_s(x):s\in[-t,0]\}$ such that $\Phi_0(x)\in L_1$ and $\Phi_{-t}(x)\in L_2$. We will call $-t$ the length of the Liouville chord.
		\item	A (positive) Liouville chord $\{\Phi_s(x):s\in[0,t]\}$ will be called essential if $f_2(\Phi_t(x))-e^tf_1(\Phi_0(x))\geq 0$.
		\item 	A negative Liouville chord $\{\Phi_s(x):s\in[-t,0]\}$ will be called essential if $f_2(\Phi_{-t}(x))-e^{-t}f_1(\Phi_0(x)))\leq 0$.
		\item A Liouville chord of $L_1$ is a trajectory $\{\Phi_s(x):s\in[0,t]\}$ such that $\Phi_0(x),\Phi_{t}(x)\in L_1$ for some $t>0$. We will call $t$ the length of the Liouville chord.
		\item 	A Liouville chord $\{\Phi_s(x):s\in[0,t]\}$ of $L_1$ will be called essential if $f_1(\Phi_t(x))-e^tf_1(\Phi_0(x))\geq 0$.
	\end{enumerate}
\end{Def}
\begin{Rem}
A Liouville chord of $L_1$ is essential if its extremities do not verify the MVT inequality given in theorem \ref{thm3}. Indeed, for $x=(q,p)$, $\Phi_t(x)=(q,e^tp)$ and therefore: \[\frac{\ln(f_1(q,e^tp))-\ln(f_1(q,p))}{\ln(e^t)}\geq 1\]
\end{Rem}
Let us now apply the previous theorems and corollaries to the examples given in section \ref{somexample}. First, notice that that:

\begin{Rem}
	In the construction given in the section \ref{somexample}, we can push a Legendrian at infinity with the Reeb flow. Therefore, the chords along which the MVT needs to be checked can be made to be arbitrarily close to the Reeb chords. Indeed, the Liouville chords are given by the flow of $\sum p_i\partial_{p_i}+s\partial_s\approx s\partial_s$ when $s>>\sum |p_i|$, for any point $(q,p,\theta,s)\in T^*M\times\Sph\times\R=T^*(M\times\Sph)$.
\end{Rem}

Therefore, as a corollary of the previous results, we get back the (classical) result:

\begin{Cor}
	Let $L$ be a connected closed Legendrian of $J^1M$ and $L\times\mathbb{S}^1$ be its lift. If $L$ does not have any Reeb chord then the projection induces a simple homotopy equivalence between $L\times\mathbb{S}^1$ and $M\times\mathbb{S}^1$. In particular, the projection of $J^1M$ on $M$ induces simple homotopy equivalence between $L$ and $M$.
\end{Cor}

Do note however that if $L$ has Reeb chords, the previous results do not yield additional information, as we have the following remark that is a direct consequence of the construction given in the second section:

\begin{Rem}
All the Liouville chords of $L\times\mathbb{S}^1$ are essential.
\end{Rem}

Notice that the two previous corollaries (and the remark) have some interesting consequence: if the projection does not induce a homotopy equivalence between the Legendrian and the $0$-section, then its exact Lagrangian lift to $T^*(M\times\Sph)$ must always have an essential Liouville chord. Therefore, at least some of the essential Liouville chords must survive displacement by any (lift of) Legendrian isotopy and any Hamiltonian isotopy of $\lcs$ type. Moreover, since Reeb chords (almost) lift to families of Liouville chords, some traces of them must survive after a Hamiltonian isotopy of $\lcs$ type. It therefore seems like essential Liouville chords are a $\lcs$ version of Reeb chords.\newline

Another way of considering essential Liouville chords stem from the following remark:
\begin{Rem}
	Let $\Lambda\times\mathbb{S}^1$ be the lift of a connected closed Legendrian $\Lambda$ as given by the construction \ref{const1}. Then $\Lambda\times\mathbb{S}^1$ has a Liouville chord from $(x,\sigma,s)\in(T^*M\times T^*_\sigma\mathbb{R})$ to $(x',\sigma,s')\in(T^*M\times T^*_\sigma\mathbb{R})$ if and only if   $(x,-s)\in J^1M$ and $(x',-s')\in J^1M$ are on the same orbit of the flow of $p\partial_p+s\partial_s$ in $J^1M$. Moreover, notice that $(x,-s)$ and $(x',-s')$ are $\Lambda$.
\end{Rem}
This leads us to the following lemma:
\begin{Lem}
	Let $\Lambda$ be a connected closed Legendrian of $(J^1M,\alpha)$ where $\alpha$ is the canonical contact form. Assume that $\Lambda$ is in $T^*M\times[\epsilon;+\infty)$ for some $\epsilon>0$, and denote by $s$ the coordinate in $\mathbb{R}$. Then, for any contact form $\alpha'$ that restricts to $\frac{\alpha}{s}$ on $T^*M\times[\epsilon;+\infty)$, $\Lambda$ is a Legendrian of $(J^1M,\alpha')$ and $\Lambda$ has a Reeb chord (for the contact form $\alpha'$) from $(x,s)$ to $(x',s')$ if and only if the exact Lagrangian lift $\Lambda\times\mathbb{S}^1$ given by construction \ref{const1} has an essential Liouville chord from $(x,\theta,-s)$ to $(x',\theta,-s')$, for any $\theta$.
\end{Lem}
\begin{proof}
	Given he previous remark, it is sufficient to check that $R:=p\partial_p+s\partial_s$ is a Reeb vector field for $\frac{\alpha}{s}$. We immediately have that $\frac{\alpha}{s}(R)=1$. Moreover, $d\frac{\alpha}{s}=d\frac{ds}{s}-d\frac{pdq}{s}=\frac{pds\wedge dq}{s^2}-\frac{dp\wedge dq}{s}$, implying that $\iota_Rd\frac{\alpha}{s}=0$.
\end{proof}
Therefore, those Lagrangian lifts enable us to study Legendrians of $J^1M$ for a family of contact forms. Indeed, given any contact form $\alpha'$ verifying the conditions of the precious lemma, and a Legendrian $\Lambda$ of $(J^1M,\alpha')$, we can study the persistence of Reeb chords of $\Lambda$ by first finding a Legendrian isotopy $\phi_t$ such that $\phi_t(\Lambda)$ verifies the conditions of the lemma, and then by studying the Liouville chords of the lift.

Finally, note that Hamiltonian isotopies of $\lcs$ type are not, in general, lifts of Legendrian isotopies and vice-versa. This therefore enables us to study Reeb/Liouville chords under the action of a larger class of deformations.

\bibliographystyle{plain}
\bibliography{./biblio.bib}
\end{document}